\documentclass[times]{article}

\usepackage[margin=3.3cm]{geometry}
\usepackage{moreverb}

\usepackage[colorlinks,bookmarksopen,bookmarksnumbered,citecolor=red,urlcolor=red]{hyperref}

\usepackage{amsthm}
\usepackage{amsmath}
\usepackage{amssymb}
\usepackage{tikz,pgfplots}
\usepackage{subcaption}
\usepackage{multirow}
\usepackage{enumerate}
\usepackage{enumitem}
\usepackage{stmaryrd}
\usepackage{cancel}

\def\nref#1{{\rm (\ref{#1})}}

\theoremstyle{plain}
\newtheorem{theorem}{Theorem}[section]

\newtheorem{definition}{Definition}[section]

\newtheorem{remark}{Remark} [section]

\DeclareMathOperator*{\argmin}{argmin}

\newcommand{\HDD}{\ensuremath{\mathbf{H}_{\text{DD}}}}
\newcommand{\rhs}{\ensuremath{\mathbf{b}}}

\newcommand{\matid}{\ensuremath{\mathbf{{I}}}}
\newcommand{\bu}{\ensuremath{\mathbf{u}}}
\newcommand{\by}{\ensuremath{\mathbf{y}}}

\newcommand{\bP}{\ensuremath{\mathbf{P}}}
\newcommand{\bB}{\ensuremath{\mathbf{B}}}

\newcommand{\bM}{\ensuremath{\mathbf{M}}}
\newcommand{\bN}{\ensuremath{\mathbf{N}}}
\newcommand{\bW}{\ensuremath{\mathbf{W}}}
\newcommand{\bZ}{\ensuremath{\mathbf{Z}}}

\newcommand{\bQ}{\ensuremath{\mathbf{Q}}}

\newcommand{\bY}{\ensuremath{\mathbf{Y}}}

\newcommand{\bv}{\ensuremath{\mathbf{v}}}
\newcommand{\bx}{\ensuremath{\mathbf{x}}}
\newcommand{\bb}{\ensuremath{\mathbf{b}}}

\newcommand{\bA}{\ensuremath{\mathbf{A}}}

\newcommand{\bH}{\ensuremath{\mathbf{H}}}
\newcommand{\br}{\ensuremath{\mathbf{r}}}

\newcommand{\bz}{\ensuremath{\mathbf{z}}}

\newcommand{\range}{\ensuremath{\operatorname{range}}}

\newcommand{\revised}[1]{{#1}}

\usepackage{stmaryrd}

\usetikzlibrary{fit,calc}

\title{Improved Polynomial Bounds and Acceleration of GMRES\\ by Solving a min-max Problem on Rectangles, and by Deflating.\footnote{This version dated \today}}

\author{Nicole Spillane\thanks{CNRS, CMAP, Ecole Polytechnique, Institut Polytechnique de Paris, 91128 Palaiseau Cedex, France (\textit{nicole.spillane@cmap.polytechnique.fr}). This work was supported in part by the ANR JCJC project DARK (research grant ANR-24-CE46-1633).}   
~and Daniel B. Szyld\thanks{Department of Mathematics, Temple University, Philadelphia, PA 19122, USA (\textit{szyld@temple.edu})} 
}
\begin{document}

\date{}

\maketitle

\textbf{Keywords: } GMRES bounds, Min-Max problem on rectangle, Weighted GMRES, Deflated GMRES,
left preconditioning, right preconditioning, split preconditioning

\textbf{AMS Subject Classification: } 65F10, 65Y05, 68W40

\pagestyle{myheadings}
\thispagestyle{plain}
\markboth{Nicole SPILLANE and Daniel B. SZYLD}{Improved GMRES Bounds}

\abstract{
Polynomial convergence bounds are considered for left, right, and split preconditioned GMRES. They
include the cases of Weighted and Deflated GMRES for a linear system \mbox{$\bA\bx=\bb$}. 
In particular, the case of positive definite $\bA$ is considered. The well-known polynomial bounds are 
generalized to the cases considered, and then reduced to solving a min-max problem on rectangles on the complex plane. Several approaches are considered and compared. The new bounds can be improved by using specific deflation spaces and preconditioners. This in turn accelerates the convergence of GMRES. Numerical examples illustrate the results obtained.}

\section{Introduction}
Let $\bA \in \mathbb C^{n \times n}$ be a non-singular matrix. Our focus is on solving a large sparse linear system of the form  
\begin{equation} \label{linsys:eq}
\bA \bx = \mathbf b. 
\end{equation}

We consider weighted GMRES with split preconditioning by $\bH_L$ on the left and $\bH_R$ on the right. We refer to this as GMRES preconditioned by $(\bH_L, \bH_R)$. What is meant is that GMRES is applied to 
\begin{equation} \label{prec:eq}
\bH_L \bA \bH_R \bu = \revised{ \bH_L} \mathbf b \mbox{ ~ and~} \bx = \bH_R \bu. 
\end{equation}
It is assumed that $\bA$, $\bH_L$ and $\bH_R$ are all non-singular. 
Throughout, we use the notation $\bH= \bH_R \bH_L$ for the combined preconditioner. The term weighted GMRES \cite{zbMATH01268410} corresponds to the use of a non-standard inner product within GMRES, here the inner product induced by some Hermitian positive definite (hpd) \mbox{matrix $\bW$.} Weighted GMRES is studied in \cite{embreedeflating17,zbMATH06376430}; \revised{see also \cite{embree2025extending}
where the different inner products are used to study the convergence of GMRES in the 2-norm.}
As we show in detail in  Section~\ref{gen.case:sec} one bound for the 
(ideal) preconditioned and weighted GMRES residual is
\begin{equation} \label{first_bound:eq}
\| \bH_L \br_k  \|_\bW \leq \min\limits_{q \in \mathbb P_k; \, q(0) = 1} \{ \|q(\bH_L \bA \bH_R)\|_\bW \}\|\bH_L \br_0 \|_{\bW},
\end{equation}
where $\br_0 = \bb - \bA \bx_0$ is the initial residual and $\mathbb P_k$ is the set of polynomials of degree at 
\mbox{most $k$.}  Our goal in this paper is to obtain good bounds for the minimization problem in \eqref{first_bound:eq}
and, at the same time, look for appropriate preconditioners and weighting matrices 
(as well as deflation spaces) to obtain better bounds. To this end, also in Section~\ref{gen.case:sec}, 
we use the Crouzeix-Palencia result \cite{zbMATH06760188}, and obtain
\[
\min\limits_{q \in \mathbb P_k; \, q(0) = 1} \{ \|q(\bH_L \bA \bH_R)\|_\bW \}
\leq (1+ \sqrt 2) \min\limits_{q \in \mathbb P_k; \, q(0) = 1} \max\limits_{z \in FOV^\bW(\bH_L\bA\bH_R)} |q(z)|, 
\] 
where $FOV^\bW$ stands for the field of values in the $\bW$-inner product. 

We also consider Deflated GMRES beginning in Section~\ref{deflated:sec}; see, e.g., 
\cite{zbMATH01729149,ggln2013,SdSK.2020-GAMM,tang2009comparison} and references therein for a full description. Recent contributions to deflation for GMRES include \cite{arXiv:2502.18317,arXiv:2005.03070}, as well as our own work \cite{zbMATH07931226}.
In Section~\ref{general.deflated:sec}, we study bounds for Deflated GMRES, and we obtain a similar min-max problem as above, but for a different FOV-type set in $\mathbb C$.

\revised{We are able to bound the sets in $\mathbb C$ for each min-max problem by rectangles, by appropriately choosing
the preconditioner, the weighting matrix $\bW$, and the deflation space.} 
In fact, we consider two different deflation spaces (in Subsections~\ref{sec:firstDS} and~\ref{sec:secondDS}), and for each of them, we develop
convergence bounds based on the min-max problem in a rectangle.
Furthermore, these rectangles are parametrized by a threshold parameter $\tau$ defining our deflation spaces.  Section~\ref{sec:minmaxrec} makes the analysis complete by discussing the solution to the min-max problem on rectangles. The new bounds are an improvement over the linear bounds in \cite{zbMATH07860856} and \cite{zbMATH07931226} which inspired the present work.
This study allows us to choose an appropriate preconditioner, deflation space, and
weight matrix for GMRES for specific problems.
This is illustrated numerically in \mbox{Section ~\ref{sec:numerical}.}

To summarize, our contributions in this article consist in a full convergence analysis for 
 split preconditioned GMRES with a non-standard inner product and deflation. In particular the deflation space does not come from approximating eigenvectors of $\bA$. When $\bA$ is positive definite, 
\revised{{\em i.e.}, when the Hermitian part of $\bA$ is hpd},
the bound depends on a min-max problem on a rectangle in the complex plane that is itself approximated by several methods. Finally, we propose to improve the bound and accelerate convergence by deflation. 

\section{GMRES convergence bounds}
\label{sec:1}
\subsection{The general Case}\label{gen.case:sec}

We begin by stating the minimization property of GMRES in the general case with a 
combined preconditioner $\bH$, and with a Hermitian positive definite (hpd) weight matrix $\bW$.

\begin{theorem}[Minimization property]
\label{th:cv}
Consider $\bW$-weighted GMRES applied to $\bA \bx = \rhs$ and preconditioned by $(\bH_L,\bH_R)$. For any $k \in \llbracket 0, n \rrbracket$, the $k$-th iterate is characterized by 
\[
\bx_k = \argmin\limits_{\bx \in \bx_0 + \mathcal K_k(\bH  \bA, \bH \br_0)}\{\|\bH_L (\mathbf b - \bA\bx)\|_{\bW}\}, 
\]
where
\[
\bH := \bH_R \bH_L
\text{ and }
 \mathcal K_k(\bH  \bA, \bH  \br_0) := \operatorname{span} \left\{\bH  \br_0, \bH  \bA \bH \br_0, \dots,  (\bH \bA)^{k-1} \bH  \br_0 \right\}.
\]
Equivalently, the $k$-th residual, $\br_k = \bb - \bA \bx_k$, satisfies
\begin{equation}
\label{eq:minHMri}
\| \bH_L \br_k  \|_\bW = \min\limits_{q \in \mathbb P_k; \, q(0) = 1} \{ \|\bH_L q(\bA \bH) \br_0 \|_{\bW} \} \leq \min\limits_{q \in \mathbb P_k; \, q(0) = 1} \{ \|q(\bH_L \bA \bH_R)\|_\bW \}\|\bH_L \br_0 \|_{\bW}, 
\end{equation}
where $\br_0 = \bb - \bA \bx_0$ and $\mathbb P_k$ is the set of polynomials of degree at most $k$. 
\end{theorem}
\begin{proof}
\revised{At step $k$, GMRES applied to \nref{prec:eq} in the $W$-inner product with initial residual $\br_0$
finds $\bH_R^{-1} x_k$, which is the
minimizer of $\|\bH_L (\mathbf b - \bA\bx_k)\|_{\bW} = \|\bH_L \br_k \|_{\bW}$ for all possible
elements $\bx_k$ of the Krylov subspace  $ \bx_0+\mathcal K_k(\bH  \bA, \bH \br_0) =  \bx_0+\mathcal K_k(\bH_L \bA \bH_R, \bH_L \br_0)$,
where $\br_k = \mathbf b - \bA\bx_k$.
Then, the minimization over the Krylov subspace and the set of polynomials follow.} To obtain the
inequality in \eqref{eq:minHMri}, first note that by definition of the matrix norm,
\[ 
 \|\bH_L q(\bA \bH) \br_0 \|_{\bW}  \leq 
 \|\bH_L q(\bA \bH)\bH_L^{-1}\|_\bW \|\bH_L \br_0 \|_{\bW}, 
\]
and that
\[
  \|\bH_L q(\bA \bH)\bH_L^{-1}\|_\bW  = \|q(\bH_L \bA \bH_R)\|_\bW.
\]
Indeed, it holds that 
$\bH_L (\bA \bH)^0 \bH_L^{-1} = \matid =  (\bH_L \bA \bH_R)^0, \quad \bH_L (\bA \bH) \bH_L^{-1} = \bH_L \bA \bH_R$,\\ $ \bH_L (\bA \bH)^2 \bH_L^{-1} = \bH_L \bA \bH_R \bH_L \bA \bH_R = (\bH_L \bA \bH_R)^2$,  \textit{etc}.
\end{proof}

\begin{remark}
The Krylov subspace does not depend on the choice of left, right or split preconditioning so long as the combined preconditioner 
\revised{$\bH = \bH_R \bH_L$} remains the same. What this choice does influence is the norm that gets minimized. So in that sense, moving (multiplicatively) some preconditioning from right to left has the same effect as weighting.  
\end{remark}

We point out that throughout this article exact arithmetic is assumed. Backward stability analysis as in \cite{zbMATH00827429,zbMATH05029305}, the influence of a perturbed operator as in \cite{zbMATH02030238}, or a perturbed deflation operator as in \cite{zbMATH06238076} have not yet been considered.

\revised{A practical approach for bounding the convergence of GMRES is to consider ideal GMRES, \textit{i.e.}, to find a bound for $\min\limits_{q \in \mathbb P_i; \, q(0) = 1} \{ \|\bH_L q(\bA \bH)\bH_L^{-1}\|_\bW \}$. 
This approximation is discussed in~\cite{zbMATH01200984} as one of six steps that go into the analysis of GMRES; see \cite{Toh97} for the difference between ideal and worst-case GMRES.}

If $\bH_L \bA \bH_R $ is diagonalizable, the term $\min\limits_{q \in \mathbb P_i; \, q(0) = 1} \{ \|q(\bH_L \bA \bH_R)\|_\bW \}$ can be bounded by a min-max problem where the max is taken over a set that contains the spectrum of  $\bH \bA$. For non-normal matrices, a leading constant in the bound appears with the conditioning of eigenvectors of $\bH \bA$ in it. Instead, we state here a convergence result where the GMRES min-max problem is posed over the field of values of the coefficient matrix $\bA$ of the linear system. This result, with the leading constant $(1 + \sqrt{2})$, is given in  \cite{zbMATH06760188} where it is shown that the field of values is a $(1+\sqrt{2})$-spectral set. See, also \cite[Equation (FOV), page 5]{arXiv:2209.01231} with context and related work. There are also extensions of this bound in \cite{zbMATH07122453}.   

In \cite[(1), (3) and Theorem 3.1]{zbMATH06760188}, it is proved (in particular) that for any bounded linear operator $A$ in a complex 
Hilbert space $(\mathcal H, \langle, \rangle_\mathcal{H}, \| \|_{\mathcal H})$ and for any rational function $f$
\revised{with no poles in $FOV^{\mathcal H}(A)$} defined below, 
\begin{equation}
\label{CP:lemmaHilbert}
\|f(A)\|_{\mathcal H} \leq (1 + \sqrt 2)  \sup\limits_{z \in FOV^{\mathcal H}(A)} |f(z)|; \quad FOV^{\mathcal H}(A) = 
\left\{ \langle A \bv, \bv\rangle_{\mathcal H} ; \, \bv\in \mathcal H, \, \| \bv \|_{\mathcal H} = 1 \right\}.
\end{equation} 

In the Hilbert space, $(\mathbb C^{n}, \langle, \rangle_\bW, \| \|_\bW)$, we obtain, for any matrix $\bB \in \mathbb C^{n \times n}$ and polynomial $q$, that 
\begin{equation}
\label{CP:lemma}
\|q(\bB)\|_\bW \leq (1 + \sqrt 2)  \sup\limits_{q \in FOV^\bW(\bB)} |q(z)|,
\end{equation} 
where the field of values of $\bB \in \mathbb C^{n \times n}$ in the $\langle \cdot, \cdot \rangle_\bW$ inner product  is the set
\[
FOV^\bW(\bB) := \left\{ \frac{\langle \bB \bz, \bz\rangle_\bW}{\langle \bz, \bz\rangle_\bW}; \, \bz\in \mathbb C^n \setminus\{0\}\right\}.
\]

\begin{remark}
The field of values of $\bB$ is also, and equally often, called the numerical range of $\bB$. 
By the well-known Hausdorff-Toeplitz theorem, see, e.g., \cite{ch.davis.71}, the field of values is convex and compact. 
\end{remark}

\begin{theorem}[Crouzeix-Palencia bound]
\label{th:CPbound}
The $k$-th residual of $\bW$-weighted GMRES applied to $\bA \bx = \rhs$ and  preconditioned by $(\bH_L,\bH_R)$ is bounded by
\[
\frac{\|\bH_L\br_k\|_\bW}{\|\bH_L \br_0 \|_\bW} \leq (1+ \sqrt 2) \min\limits_{q \in \mathbb P_k; \, q(0) = 1} \max\limits_{z \in FOV^\bW(\bH_L\bA\bH_R)} |q(z)|. 
\]  
Moreover, $ FOV^\bW(\bH_L\bA\bH_R) = FOV^{\bH_L^* \bW \bH_L}(\bA \bH)$, where again $\bH = \bH_R \bH_L$. 
\end{theorem}
\begin{proof}
In the residual bound \eqref{eq:minHMri}, we apply the Crouzeix-Palencia result \eqref{CP:lemma} to get
\begin{align*} 
\| q(\bH_L \bA \bH_R)\|_\bW & \leq (1+\sqrt{2})\sup\limits_{z \in FOV^\bW(\bH_L \bA \bH_R)} |q(z)| 
=  (1+\sqrt{2})\max\limits_{z \in FOV^\bW(\bH_L \bA \bH_R)} |q(z)|. 
\end{align*}
We have replaced the $\sup$ by $\max$ as we are maximizing the 
\revised{uniform norm of a polynomial} over a compact set in a finite dimensional space. The equality between the two weighted field of values comes from their definition. 
\end{proof}

In order to solve the min-max problem in Theorem~\ref{th:CPbound}, we must characterize the set 
\[
FOV^{\bH_L^* \bW \bH_L}(\bA \bH) = \left\{ \frac{\langle \bA \bH \bz, \bz\rangle_{\bH_L^* \bW \bH_L}}{\langle \bz, \bz\rangle_{\bH_L^* \bW \bH_L}}; \, \bz\in \mathbb C^n\setminus\{0\} \right\}.  
\]

\subsection{The case when $\bA$ is positive definite (pd) and $\bH = \bH_L^* \bW \bH_L$} 
\label{sub:ApdHnorm}

Assume that $\bA$ is pd, \textit{i.e.}, that its Hermitian part
is pd, $\bH$ is hpd, and the residual is minimized in the $\bH$-norm, \textit{i.e.},
 $\bH = \bH_L^* \bW \bH_L$, 
so that $\| \bH_L \br_i  \|_\bW = \| \br_i  \|_\bH$ in \eqref{eq:minHMri}. Particular cases when this occurs include the following:
\begin{itemize}
\item Left preconditioning by hpd $\bH$ and $\bW = \bH^{-1}$, 
\item Right preconditioning by hpd $\bH$ and $\bW = \bH$, 
\item \revised{Split preconditioning by $\bH_L^*$ on the right, $\bH_L$ on the left, and $\bW = \matid$.} 
\end{itemize}

The idea to apply GMRES in the inner product induced by the preconditioner has also been explored by \cite{zbMATH01201042,zbMATH07860856,zbMATH01096035}. Moreover the authors of \cite{zbMATH06290704} propose other combined choices of preconditioner and weighted inner product.

Let $\bM$ and $\bN$ denote respectively the Hermitian and skew-Hermitian parts of $\bA$:
\begin{equation}
\label{eq:splitA}
\bA = \bM + \bN, \, \bM := \frac{\bA + \bA^*}{2} \text{ and } \bN := \frac{\bA - \bA^*}{2} \cdot
\end{equation}

The eigenvalues of $\bH \bM$ are real and positive as a result of $\bH$ and $\bM$ 
\revised{being hpd}. The eigenvalues of $\bN \bH$ and of $\bM^{-1}\bN$ are purely imaginary as a result of $\bH$ and $\bM$ being hpd, and of $\bN$ begin skew-Hermitian. Moreover, the non-zero eigenvalues of $\bN \bH$ and of $\bM^{-1}\bN$ come in complex-conjugate pairs. In what follows, $\lambda_{\min}(\bH\bM), \lambda_{\max}(\bH \bM) \in \mathbb R^+ \setminus\{ 0 \}$ denote the extreme eigenvalues of $\bH\bM$, and $\rho(\cdot)$ denotes the spectral radius of a matrix (the modulus of the eigenvalue of largest modulus). With these assumptions and notation, 
\begin{align*}
FOV^\bH(\bA \bH) &=  \left\{ \frac{\langle \bH \bA \bH \bz, \bz\rangle}{\langle \bH \bz, \bz\rangle}; \, \bz\in \mathbb C^n\setminus\{0\} \right\}\\
&  \subset \underbrace{\left\{ \frac{\langle \bH \bM \bH \bz, \bz\rangle}{\langle \bH \bz, \bz\rangle}; \, \bz\in \mathbb C^n\setminus\{0\} \right\}}_{\revised{\subset} \mathbb R} +  \underbrace{\left\{ \frac{\langle \bH \bN \bH \bz, \bz\rangle}{\langle \bH \bz, \bz\rangle}; \, \bz\in \mathbb C^n\setminus\{0\} \right\}}_{\revised{\subset} i  \mathbb R}\\
 &  = {\left\{ \frac{\langle \bM \bz, \bz\rangle}{\langle \bH^{-1} \bz, \bz\rangle}; \, \bz\in \mathbb C^n\setminus\{0\} \right\}} +  {\left\{ \frac{\langle \bN \bz, \bz\rangle}{\langle \bH^{-1} \bz, \bz\rangle}; \, \bz\in \mathbb C^n\setminus\{0\} \right\}}.
\end{align*}
\revised{
Taking $\bH^{1/2}$ to be the square root of $\bH$ and  $\by = \bH^{-1/2} \bz$, the fractions can be rewritten as
\[
\frac{\langle \bH^{1/2} \bM \bH^{1/2} \by, \by\rangle}{\langle \by, \by\rangle} \text{ and }  \frac{\langle \bH^{1/2} \bN  \bH^{1/2} \by, \by\rangle}{\langle \by, \by\rangle}, 
\]
where we recognize Rayleigh quotients for the hpd matrix $\bH^{1/2} \bM \bH^{1/2}$ and the skew-Hermitian matrix $\bH^{1/2} \bN \bH^{1/2}$. Since these matrices are similar, respectively, to $\bH \bM$ and $\bH \bN$, it can be concluded that  
\begin{equation}
\label{Omega_best_inclusion:eq}
FOV^\bH(\bA \bH) \subset  [ \lambda_{\min} (\bH \bM), \lambda_{\max} (\bH \bM) ] + i [-  \rho(\bN \bH) , \rho(\bN \bH) ].
\end{equation}}


We have thus presented a rectangle in $\mathbb C$ that contains the field of values over which the
min-max problem is defined.
Another rectangle can be obtained by a technique used 
in our two previous works \cite{zbMATH07860856,zbMATH07931226}, namely, multiply the fraction in the imaginary term by $\displaystyle{\frac{\langle \bM \bz, \bz\rangle}{\langle \bM \bz, \bz\rangle} \cdot}$
Then, for any $\bz \in \mathbb C^n \setminus \{0\}$, by a similar argument
\begin{equation} \label{Omega_inclusion:eq}
\frac{\langle \bN \bz, \bz\rangle}{\langle \bH^{-1} \bz, \bz\rangle} = \frac{\langle \bN \bz, \bz\rangle}{\langle \bM \bz, \bz\rangle} \times \frac{\langle \bM \bz, \bz\rangle}{\langle \bH^{-1} \bz, \bz\rangle} \in i [- \rho(\bM^{-1} \bN)\lambda_{\max}(\bH \bM), \rho(\bM^{-1} \bN)\lambda_{\max}(\bH \bM)] .
\end{equation}
This gives us the inclusion 
\[
FOV^\bH(\bA \bH) \subset [\lambda_{\min}(\bH \bM), \lambda_{\max}(\bH \bM)] + i  [- \rho(\bM^{-1} \bN)\lambda_{\max}(\bH \bM), \rho(\bM^{-1} \bN)\lambda_{\max}(\bH \bM)]. 
\]
\revised{It is necessarily less strict than the previous inclusion~\nref{Omega_best_inclusion:eq}. Indeed, the inclusion $\frac{\langle \bN \bz, \bz\rangle}{\langle \bH^{-1} \bz, \bz\rangle} \in  i [- \rho(\bH \bN), \rho(\bH \bN)]$ is tight, whereas~\nref{Omega_inclusion:eq} is in general not (each term in the product having been bounded independently).} 
 We summarize these findings in a theorem. 

\begin{theorem}
\label{th:CP}
Consider $\bW$-weighted GMRES applied to $\bA \bx = \rhs$ and preconditioned by $(\bH_L,\bH_R)$. Under the three conditions that $\bA$ is pd, $\bH:=\bH_R \bH_L$ is hpd and $\bH = \bH_L^* \bW \bH_L$, 
 the $k$-th residual is bounded by
\[
\frac{\|\br_k\|_\bH}{\|\br_0 \|_\bH} \leq (1+ \sqrt 2) \min\limits_{q \in \mathbb P_k; \, q(0) = 1} \max\limits_{z \in FOV^\bH(\bA\bH)} |q(z)|; \quad FOV^\bH(\bA \bH) \subset \Omega_1 \subset \Omega_2,  
\]  
where
\begin{equation}
\label{eq:Omega1}
\Omega_1 :=  [ \lambda_{\min} (\bH \bM), \lambda_{\max} (\bH \bM) ] + i [-  \rho(\bN \bH) , \rho(\bN \bH) ]. 
\end{equation}
and
\begin{equation}
\label{eq:Omega2}
\Omega_2 :=  [\lambda_{\min}(\bH \bM), \lambda_{\max}(\bH \bM)] + i  [- \rho(\bM^{-1} \bN)\lambda_{\max}(\bH \bM), \rho(\bM^{-1} \bN)\lambda_{\max}(\bH \bM)]. 
\end{equation}
(Recall that $[\lambda_{\min} (\bH \bM),  \lambda_{\max} (\bH \bM)]$ is a real positive interval that contains all eigenvalues of $\bH\bM$, and $\rho(\cdot)$ denotes the spectral radius of a matrix.) 
\end{theorem}

These bounds depend on
\begin{itemize}
\item how well the Hermitian part $\bM$ of $\bA$ is preconditioned by $\bH$ \textit{via} $\lambda_{\min}(\bH \bM)$ and $\lambda_{\max}(\bH \bM)$,
\item and \begin{itemize} 
\item either, on how well the skew-Hermitian part $\bN$ of $\bA$ is preconditioned by $\bH$  \textit{via} $ \rho(\bN \bH)$, 
\item or, on how non-Hermitian the problem is, \textit{via}, $\rho(\bM^{-1} \bN)$.
\end{itemize}
\end{itemize} 

We defer the solution of the min-max problem on $\Omega_1$ and $ \Omega_2$ to Section~\ref{sec:minmaxrec}. First we consider deflated GMRES.  

\section{Deflated GMRES} \label{deflated:sec}

\subsection{The general case} \label{general.deflated:sec}

We temporarily relax the assumptions from the previous subsection. Let $\bA$, $\bH_R$, $\bH_L$ be non-singular 
$n \times n$ matrices, $\bW$ be hpd, and  $\bH := \bH_R \bH_L$. 

\begin{definition}
\label{def:PDQD}
Let $\bY, \bZ \in \mathbb C^{n\times m}$ be two full rank matrices. Under the assumption that $\ker(\bY^*) \cap \range(\bA \bZ) = \{ \mathbf{0}\}$, let
\begin{equation}
\label{eq:defPD}
\bP_D := \matid - \bA \bZ(\bY^* \bA \bZ)^{-1} \bY^* \text{ and }  \bQ_D := \matid - \bZ (\bY^* \bA \bZ)^{-1}\bY^* \bA.
\end{equation}
These are projection operators called the deflation operators.  
\end{definition}
In \cite{SdSK.2020-GAMM}, the projectors  $\bP_D$ and $\bQ_D$ are called
\textit{sibling projectors}.
What is meant is that defining one also defines the other unambiguously. Deflated GMRES is the application of GMRES to the singular system $\bP _D \bA \bx = \bP_D \bb$. 
 Deflation can be applied simultaneously with preconditioning by $(\bH_L, \bH_R)$ in which case GMRES is applied to 
\[
\bH_L \bP _D \bA \bH_R \bu = \bH_L \bP_D \bb; \quad  \bx = \bH_R \bu. 
\]
Assume that $\bY^* \bH^{-1} \bZ$ is non-singular, then GMRES does not break down; see \cite[Theorem 3.5]{zbMATH07152809} for left preconditioning and \cite[Theorem 3.2]{zbMATH07931226} for right preconditioning. This property remains true for GMRES in any weighted inner product. Moreover, the residuals of $\bW$-weighted and $\bP_D$-deflated GMRES applied to $\bA \bx = \bb$  with preconditioning by $(\bH_L, \bH_R)$  follow the characterization from Theorem~\ref{th:cv}, \textit{i.e.}, $\br_k = \bP_D(\bb - \bA \bx_k)$, satisfies
\begin{align*}
\| \bH_L \br_k  \|_\bW  & = \min\limits_{q \in \mathbb P_k; \, q(0) = 1} \{ \|\bH_L q(\bP_D \bA \bH) \br_0 \|_{\bW} \}\\
& = \min\limits_{q \in \mathbb P_k; \, q(0) = 1} \{ \|q(\bP_D \bA \bH) \bP_D \br_0 \|_{\bH_L^* \bW \bH_L } \} \\ 
&\leq \min\limits_{q \in \mathbb P_k; \, q(0) = 1} \{ \|q(\bP_D \bA \bH) \bP_D\|_{\bH_L^* \bW \bH_L} \}\|\br_0 \|_{\bH_L^* \bW\bH_L } \\ 
&= \ \min\limits_{q \in \mathbb P_k; \, q(0) = 1} \{ \| q(\bP_D \bA \bH \bP_D)\|_{\bH_L^*\bW\bH_L} \}\|\bH_L \br_0 \|_{\bW} .
\end{align*}

If we applied the Crouzeix-Palencia bound of Theorem~\ref{th:CP} directly to  $\| q(\bP_D \bA \bH \bP_D )\|_{\bH_L^*\bW\bH_L}$, the min-max problem would be posed over the field of values of the singular operator $\bP_D \bA \bH \bP_D$, which includes $0$. This is not useful, because of the constraint that $q(0) = 1$.  Instead we apply \eqref{CP:lemmaHilbert} in the Hilbert space $(\range(\bP_D), \langle , \rangle_{\bH_L^*\bW\bH_L}, \| \|_{\bH_L^*\bW\bH_L})$ to the (bounded linear) operator in $\range(\bP_D)$ represented by the matrix $\bP_D \bA \bH \bP_D$. This gives, for any polynomial $q$, that 
\[
\|q(\bP_D \bA \bH \bP_D )\|_{\bH_L^* \bW \bH_L}\leq (1 + \sqrt 2)  \sup\limits_{\revised{z} \in FOV^{\bH_L^* \bW \bH_L}(\bP_D \bA \bH_{|\range(\bP_D)})} |q(z)|,
\]
where the $\bW$-field of values of any matrix $\bB \in \mathbb C^{n \times n}$ restricted to $V \subset \mathbb C^n$ is defined by 
\[
FOV^\bW(\bB_{|V}) := \left\{ \frac{\langle \bB \bz, \bz\rangle_\bW}{\langle \bz, \bz\rangle_\bW}; \, \bz\in V \right\}, \quad \text{ for } V \subset \mathbb C^n. 
\]
in agreement with the definition in \cite[page 268]{zbMATH06846491}. 

\begin{theorem} 
\label{th:cvdefGMRES}
Let $\bP_D$ be defined by \eqref{eq:defPD}, and assume that the two following conditions hold
\begin{equation} \label{eq:PD.conditions}
\ker(\bY^*) \cap \range(\bA \bZ) = \{ \mathbf{0}\} \text{ and }\operatorname{ker}( \bY^*  ) \cap \operatorname{range}(\bH^{-1} \bZ)= \{ \mathbf{0}\}.
\end{equation}
Then, the $k$-th residual of $\bW$-weighted GMRES applied to $\bA \bx = \rhs$, preconditioned by $(\bH_L,\bH_R)$ and deflated by $\bP_D$ is bounded by
\[
\frac{\|\br_k\|_{\bH_L^* \bW \bH_L} }{\| \br_0 \|_{\bH_L^* \bW \bH_L} } \leq  (1 + \sqrt 2)  \min\limits_{q \in \mathbb P_k; \, q(0) = 1}   \max_{z \in FOV^{\bH_L^* \bW \bH_L}\left(\bP_D \bA \bH _{| \range( \bP_D)}\right)}  \{ |q(z)|\}   , 
\]
where again $\br_k = \bP_D(\bb - \bA \bx_k)$ and  $\bH = \bH_R \bH_L$.
\end{theorem}

It remains to characterize the weighted, preconditioned and deflated field of values. 

\subsection{The case  when $\bA$ is pd and $\bH = \bH_L^* \bW \bH_L$ and $\bY = \bH \bA \bZ$} 

This is the counterpart of Section~\ref{sub:ApdHnorm} for the deflated case. Assume that $\bA$ is pd, $\bH = \bH_R \bH_L$ is hpd, and that the residual norm which is minimized (also the norm for the FOV of interest) is equal to the combined preconditioner $\bH$, \textit{i.e.}, 
$\bH_L^* \bW \bH_L = \bH$.  
 Under these condition, the min-max problem in Theorem~\ref{th:cvdefGMRES} is posed over 
\begin{equation}
\label{eq:temp}
FOV^{\bH}\left(\bP_D \bA \bH_{| \range( \bP_D)}\right)  =  \left\{ \frac{\langle \bH \bP_D \bA \bH  \bx , \bx \rangle}{\langle \bH \bx, \bx\rangle}; \bx \in \range( \bP_D) \right\}.  
\end{equation}

As in our previous work \cite{zbMATH07931226}, we assume that $\bY = \bH \bA \bZ$ in order to have a projection $\bP_D$ that is $\bH$-orthogonal,
or equivalently  
\revised{self-adjoint with respect to the 
$\bH$-inner product,
meaning that $ \bH \bP_D =  \bP_D^* \bH$,} and then
\begin{align}
FOV^{\bH}\left(\bP_D \bA \bH _{| \range( \bP_D)}\right) &=  \left\{ \frac{\langle \bH \bA \bH  \bx , \bx \rangle}  {\langle \bH \bx, \bx\rangle}; \bx \in \range( \bP_D) \right\} \notag \\ 
&\subset \underbrace{\left\{ \frac{\langle \bH \bM \bH  \bx , \bx \rangle}  {\langle \bH \bx, \bx\rangle}; \bx \in \range( \bP_D) \right\}}_{\subset \mathbb R} + \underbrace{\left\{ \frac{\langle \bH \bN \bH  \bx , \bx \rangle}  {\langle \bH \bx, \bx\rangle}; \bx \in \range( \bP_D) \right\}}_{\subset i \mathbb R} \label{eq:temp2}.
\end{align}

The plan is to assume that $\bH$ is already a good preconditioner for $\bM$ and use deflation to bound $\displaystyle{\frac{\langle \bH \bN \bH  \bx , \bx \rangle}  {\langle \bH \bx, \bx\rangle}}$ on $\range(\bP_D)$.

\begin{remark}
Equations~\eqref{eq:temp} and~\eqref{eq:temp2} hold without the assumption that $\bA$ is positive definite but the field of values may contain $0$, \textit{e.g.}, if both $\bA$ and $- \bA$ are not pd. 
\end{remark}

\begin{remark}
The combination of $\bA$ being pd and $\bY = \bH \bA \bZ$ ensures that the two conditions \eqref{eq:PD.conditions} in Theorem~\ref{th:cvdefGMRES} are satisfied as long as $\bY$ is full rank. 
\end{remark}

We will now consider two different deflation spaces.

\subsubsection{Spectral deflation space based on $\bH \bN$}
\label{sec:firstDS}
  
\begin{theorem}
\label{th:gevpHN}
Consider $\bW$-weighted GMRES applied to $\bA \bx = \rhs$, preconditioned by $(\bH_L,\bH_R)$ and deflated by $\bP_D$ (defined by~\eqref{eq:defPD}). We make the four assumptions that $\bA$ is pd, $\bH := \bH_R \bH_L$ is hpd, $\bH = \bH_L^* \bW \bH_L$ and $\bY = \bH \bA \bZ$. 

Moreover, let $(\lambda^{(k)}, \bx^{(k)})_{1 \leq k \leq n}$ be the eigenpairs of the generalized eigenvalue problem
\begin{equation}
\label{eq:gevpHN}
\bN \bx^{(k)} = \lambda^{(k)} \bH^{-1} \bx^{(k)}.
\end{equation}
If, for a given $\tau >0$, the columns of $\bY$ are set to be the vectors $\{\bx^{(k)} ; |\lambda^{(k)}| > \tau \}$, then the $k$-th residual is bounded by
\[
\frac{\|\br_k\|_\bH}{\|\br_0 \|_\bH} \leq (1+ \sqrt 2) \min\limits_{q \in \mathbb P_k; \, q(0) = 1} \max\limits_{z \in \Omega_1^\tau} | q(z)|,  
\]  
where
\begin{equation}
\label{eq:Omega1tau}
\Omega_1^\tau :=  [\lambda_{\min}(\bH \bM), \lambda_{\max}(\bH \bM)] + i  [- \tau, \tau]. 
\end{equation}
(Recall that $[\lambda_{\min} (\bH \bM),  \lambda_{\max} (\bH \bM)]$ is the real positive interval that contains all eigenvalues of $\bH\bM$). 
\end{theorem}
\begin{proof}
Since $\bN$ is skew-Hermitian and $\bH^{-1}$ is Hermitian positive definite, the eigenvectors $\bx^{(k)}$ can be chosen to form a $\bH^{-1}$-orthonormal basis of $\mathbb C^n$ and $\{\bx^{(k)} ; |\lambda^{(k)}| > \tau \} \perp^{\bH^{-1}} \{\bx^{(k)} ; |\lambda^{(k)}| \leq \tau \}$ (see \cite[Lemma 2.2]{zbMATH07931226}). Consequently, $\range ( \bH \bP_D) = \operatorname{span} \{\bx^{(k)} ; |\lambda^{(k)}| \leq \tau \}$ since 
\[
\range (\bP_D) = \range(\bY)^\perp = \left(\operatorname{span} \{\bx^{(k)} ; |\lambda^{(k)}| > \tau \}\right)^\perp = \operatorname{span} \{\bH^{-1} \bx^{(k)} ; |\lambda^{(k)}| \leq \tau \}. 
\]
This way, in the application of Theorem~\ref{th:cvdefGMRES} with \eqref{eq:temp2}, the purely imaginary term is  
\[
\left\{ \frac{\langle \bH \bN \bH  \bx , \bx \rangle}  {\langle \bH \bx, \bx\rangle}; \bx \in \range( \bP_D) \right\} = \left\{ \frac{\langle \bN  \bx , \bx \rangle}  {\langle \bH^{-1} \bx, \bx\rangle}; \bx \in \range( \bH \bP_D)  \right\} \subset i [-\tau, \tau], 
\]
and the min-max problem can indeed be solved over $\Omega_1^\tau$.  
\end{proof}

\paragraph{Practical limitation.} 
We have assumed that $\bY = \bH \bA \bZ$ in order to make sure
\eqref{eq:PD.conditions} holds. Thus, in order to set up $\bP_D$, we need to compute $\bA^{-1} \bH^{-1} \bY $, or a different basis that spans the same space. By definition of $\bY$, $\range(\bH^{-1} \bY) = \range(\bN \bY)$. If $\bN$ is non-singular: $\bA^{-1} = (\bN (\matid + \bN^{-1} \bM))^{-1} = (\matid + \bN^{-1} \bM)^{-1} \bN^{-1}$ so that
\[
\bA^{-1} \bH^{-1} \range( \bY ) = (\matid + \bN^{-1} \bM)^{-1} \range( \bY). 
\]
It is not clear that we can efficiently compute the space spanned by the columns of $\bZ = (\matid + \bN^{-1} \bM)^{-1} \bY$.
We believe this to be a technical assumption that is not essential for the efficiency of the method. In 
the numerical results in Section~\ref{sec:numerical} we propose setting $\bZ = \bY$ or $\bZ = \bN \bY$. These choices do not satisfy $\bY = \bH\ \bA \bZ$ but they do satisfy the two conditions 
\eqref{eq:PD.conditions} in Theorem 2.1.

\subsubsection{Spectral deflation space based on $\bM^{-1} \bN$}
\label{sec:secondDS}

Another option is to compute the spectral deflation space from our previous work \cite{zbMATH07931226}.

\begin{theorem}
\label{th:gevpMinvN}
Consider $\bW$-weighted GMRES applied to $\bA \bx = \rhs$, preconditioned by $(\bH_L,\bH_R)$ and deflated by $\bP_D$ (defined by~\eqref{eq:defPD}). We make the four assumptions that $\bA$ is pd, $\bH := \bH_R \bH_L$ is hpd, $\bH = \bH_L^* \bW \bH_L$ and $\bY = \bH \bA \bZ$. 
Moreover, let $(\lambda^{(k)}, \bx^{(k)})_{1 \leq k \leq n}$ be the eigenpairs of the generalized eigenvalue problem
\begin{equation}
\label{eq:gevpMinvN}
\bN \bx^{(k)} = \lambda^{(k)} \bM \bx^{(k)}.
\end{equation}
If for a given $\tau >0$, the columns of $\bZ$ are set to be the vectors $\{\bx^{(k)} ; |\lambda^{(k)}| > \tau \}$, then the $k$-th residual is bounded by
\[
\frac{\|\br_k\|_\bH}{\|\br_0 \|_\bH} \leq (1+ \sqrt 2) \min\limits_{q \in \mathbb P_k; \, q(0) = 1} \max\limits_{z \in \Omega_2^\tau} | q(z)|,  
\]  
where
\begin{equation}
\label{eq:Omega2tau}
\Omega_2^\tau = [\lambda_{\min}(\bH \bM); \lambda_{\max}(\bH\bM) ] + i \lambda_{\max}(\bH\bM) [- \tau; \tau ]. 
\end{equation}
\end{theorem}
\begin{proof}
Since $\bN$ is skew-Hermitian and $\bM$ is Hermitian positive definite, the eigenvectors $\bx^{(k)}$ can be chosen to form an $\bM$-orthonormal basis of $\mathbb C^n$ and $\{\bx^{(k)} ; |\lambda^{(k)}| > \tau \} \perp^{\bM} \{\bx^{(k)} ; |\lambda^{(k)}| \leq \tau \}$ (see \cite[Lemma 2.2]{zbMATH07931226}). Moreover, it can be noticed that $\range(\bA \bZ) \subset \range(\bM \bZ) + \range(\bN \bZ) = \range(\bM \bZ)$ because the columns in $\bZ$ are eigenvectors of \eqref{eq:gevpMinvN}. 
 Consequently, 
\[
\range (\bP_D) = \range(\bY)^\perp =  \range(\bH \bA \bZ)^\perp ,
\]
and 
\[
\range ( \bH \bP_D) = \range(\bA \bZ)^\perp =  \range(\bM \bZ)^\perp =  \operatorname{span} \{\bx^{(k)} ; |\lambda^{(k)}| \leq \tau \}
\]
This way, in the application of Theorem~\ref{th:cvdefGMRES} with \eqref{eq:temp2}, the purely imaginary term is  
\begin{align*}
\left\{ \frac{\langle \bH \bN \bH  \bx , \bx \rangle}  {\langle \bH \bx, \bx\rangle}; \bx \in \range( \bP_D) \right\} 
&=\left\{ \frac{\langle \bH \bN \bH  \bx , \bx \rangle}{\langle \bH \bM \bH  \bx , \bx \rangle}  \frac{\langle \bH \bM \bH  \bx , \bx \rangle}{\langle \bH \bx, \bx\rangle}   ; \bx \in \range( \bP_D) \right\} \\ 
&=  \left\{ \frac{\langle \bN  \bx , \bx \rangle}{\langle \bM \bx , \bx \rangle}  \frac{\langle \bM  \bx , \bx \rangle}{\langle \bH^{-1} \bx, \bx\rangle}   ; \bx \in \range(\bH \bP_D) \right\} \\ 
&\subset  i \lambda_{\max}(\bH\bM) [- \tau; \tau ],  
\end{align*}
and the min-max problem can indeed be solved over $\Omega_2^\tau$.  
\end{proof}

\section{Solution of the min-max problem on a rectangle}
\label{sec:minmaxrec}
In our quest for improved bounds for GMRES,
we have arrived four times at the solution of the min-max problem over a rectangle of the complex plane.
Namely, 
$\Omega_1$ and  $\Omega_2$ in Theorem~\ref{th:CP}, $\Omega_1^\tau$ in Theorem~\ref{th:gevpHN}, and $\Omega_2^\tau$ in Theorem~\ref{th:gevpMinvN}.
In all these cases, the rectangle is both \revised{ on the right of the origin and symmetric around the real axis}.
In order to complete our GMRES convergence bounds, we address the solution of the min-max problem on such 
rectangular domains, \textit{i.e.}, the approximation of 
\begin{equation} \label{eq:minmaxOmega}
K_k(\Omega) := \min\limits_{q \in \mathbb P_k; \, q(0) = 1} \max\limits_{z \in \Omega} |q(z)|. 
\end{equation}
A well-known and easy to check property is that $K_k(\Omega)$  remains constant under \revised{dilation} of $\Omega$:  $K_k(\Omega) =  K_k(a \Omega)$ for any $a \in \mathbb C\setminus\{0\}$. 
Thus, without loss of generality, we can restrict ourselves to the case where 
\[
\Omega := [1, \mu] + i [-\rho, \rho], \text{ for } \mu, \, \rho \in \mathbb R^+. 
\]

We proceed by considering different approaches to find a bound for the min-max 
\mbox{problem \eqref{eq:minmaxOmega}.}

\paragraph{Elman bound.}
This is a linear bound first proved in \cite{zbMATH03831185,elman1982iterative}. For a pd matrix $\bB$ in a generic norm, the bound reads 
\[
K_k(\Omega) \leq \left[1 -\left( \frac{d(0,FOV(\bB))}{\|\bB\|} \right)^2  \right]^{k/2}, 
\]
where $d(0,FOV(\bB))$ is the distance between $0$ and the field of values of $\bB$. In our setting, we do not have an equivalent for $\|\bB\|$, however according to \cite[eq. (5.7.21)]{zbMATH06125590}), it is bounded from below with respect to $r(\bB):= \max \{ |z|; z \in FOV(\bB)\} $ (called the numerical radius of $\bB$) since
\[
r(\bB) \leq \|\bB\| \leq 2 r(\bB). 
\]
We normalize $\Omega$ with $d(0,FOV(\bB)) = d(0,\Omega)  = 1$ and $r(\bB) = \max \{ |z|; z \in \Omega \}    = \sqrt{\mu^2 + \rho^2}$ so that 
\begin{equation}
\label{eq:Elman}
K_k(\Omega) \leq \left[1 - \frac{ 1}{4 (\mu^2 + \rho^2)}  \right]^{k/2} \text{ (Elman bound)}. 
\end{equation}

\paragraph{Disk bound.}

From \cite[Theorem 5.2]{zbMATH03901905} another linear bound can be obtained. If $S$ is a non-empty compact set of $\mathbb C$ such that there exists a disk that contains $S$ but not $0$, then there is a disk $D(a,r)$ of center $a$ and radius $r$ that contains $S$ while minimizing $r/|a|$. 
In this case $ K_1(S) = {r}{/|a|}$. 
Furthermore, in \cite[(5.5) and (5.6)]{zbMATH03901905}, the case where $S$ is a segment is considered. In particular, if $S=[\alpha, \overline \alpha]$, the optimal disk is centered at $a = |\alpha|^2/ \Re (\alpha)$ and of radius $r = |a - \alpha| = |a - \overline \alpha|$ and   $
K_1([\alpha, \overline \alpha]) = |\Im(\alpha)|/{|\alpha|}.$
To come back to our rectangle $\Omega$, we first consider the segment $S = [1 - i \rho, 1+ i\rho]$ (\textit{i.e.}, $\alpha = 1-i\rho$) which is the left hand side edge of $\Omega$, and obtain
\[
a = 1 + \rho^2,\, r = | a - \alpha | = \rho \sqrt{\rho^2 + 1}, \, \text{and } K_1([1+i \rho , 1 - i \rho]) = \frac{\rho}{\sqrt{1 + \rho^2}} \cdot
\]
The same polynomial is optimal for the min-max problem posed over any compact set that both contains $[1 - i \rho, 1+ i\rho]$ and is contained in $D(a,r)$. For this reason it also holds that 
\begin{equation}\label{eq:boundOS}
K_k(\Omega) \leq \left[ \frac{\rho}{\sqrt{1 + \rho^2}} \right]^{k} \text{ if }  \mu \leq 2a -1 =  2\rho^2 + 1, 
\end{equation}
where the condition comes from enclosing $\Omega$ in the disk; see Figure~\ref{fig:geomdisksegment}--left.
 We could derive similar results for different values of $\mu$ and $\rho$. We have checked numerically that the value given by \eqref{eq:boundOS} matches the solution of the min-max problem over a vertical segment proved in \cite[Corollary (2.8)]{zbMATH03988604}, which is cited in \cite{zbMATH00839241}. 
When it applies ($\mu \leq \revised{2} \rho^2 + 1$), the linear bound \eqref{eq:boundOS} is sharper than \eqref{eq:Elman}. Indeed, 
\[
1 -\frac{ 1}{4(\mu^2 + \rho^2)} > 1 -\frac{ 1}{\mu^2 + \rho^2} = \frac{{\mu^2 + \rho^2 - 1}}{\mu^2 + \rho^2} > \frac{\rho^2}{1 + \rho^2}, 
\] 
since $z \mapsto (z-1) / z $ is increasing, and since $\mu > 1$, $\mu^2 + \rho^2  > 1+ \rho^2 $. Table~\ref{tab:comparelinears} gives some numerical values. The difference between both bounds is particularly significant when $\mu$ reaches its maximal admissible value $2 \rho^2+1$. 

\begin{table}[htb]
\centering
\begin{tabular}{c|cccc}
 $(\mu, \rho)$ &  $(2,4)$ & $ (33, 4)$& $(2,10)$ & $ (201 , 10)$  \\ 
\hline 
\eqref{eq:Elman} (Elman) & 0.9937 &   0.9999 &   0.9988 &    $1 - 3.09 \cdot 10^{-6}$ \\
\eqref{eq:boundOS} (Disk) &  0.9701 &   0.9701 &   0.9950 &   0.9950 \\
\end{tabular}
\caption{Comparison between the bounds for $K_1(\Omega)$ given by \eqref{eq:Elman} (Elman) and \eqref{eq:boundOS} (Disk). A smaller number corresponds to a better approximation.}
\label{tab:comparelinears}
\end{table}


\paragraph{Disk-segment bound.}

The 
rectangle
can be enclosed in a disk-segment (as shown in Figure~\ref{fig:geomdisksegment}--center)\footnote{\revised{In \cite{zbMATH05029264},
the field of values is enclosed in a disk-segment obviating the intermediate step of enclosing it in a rectangle. Here,
for uniformity in the analysis and comparison  we keep the rectangle. We do not expect much difference in the obtained bounds.}}
\[
\Omega \subset D= \{z \in \mathbb C; \, \Re(z) \geq  1 \text{ and } |z| \leq \sqrt{\mu^2 + \rho^2} \}. 
\]
An application of \cite[Lemma 2.2]{zbMATH05029264} (see also the discussion in \cite{embree2025extending}) gives
\begin{equation}
\label{eq:Beck}
K_k(\Omega) \leq K_k(D) \leq \min \left\{2 + \gamma_\beta , \frac{2}{1 - \gamma_\beta^{k+1}}  \right\} \gamma_\beta^k;
\end{equation}
where $\beta \in [0, \pi/2 ]$ and $\gamma_\beta$ are defined by   
\[
\operatorname{cos}(\beta) = \frac{1}{\sqrt{\mu^2 + \rho^2}}  \text{ and }\gamma_\beta = 2 \operatorname{sin} \left(\frac{\beta}{4-2 \beta/\pi} \right) < \operatorname{sin}(\beta) < 1. 
\]
(In \cite[Lemma 2.2]{zbMATH05029264}, it is also proved that $ \gamma_\beta^k \leq K_k(D) $ but this is not a lower bound for $K_k(\Omega)$.)

\begin{figure}[htb]
\begin{minipage}{0.36\textwidth}
\includegraphics[width=\textwidth]{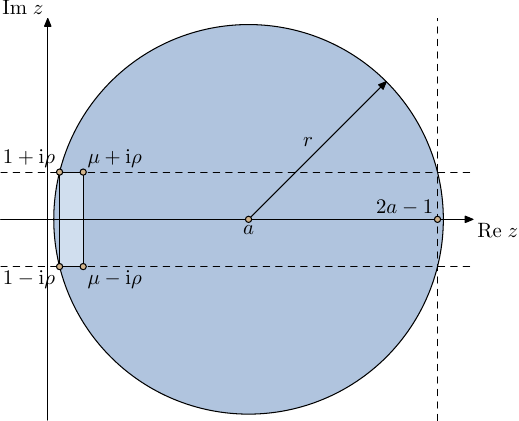}
\end{minipage}
\begin{minipage}{0.28\textwidth}
\includegraphics[width=\textwidth]{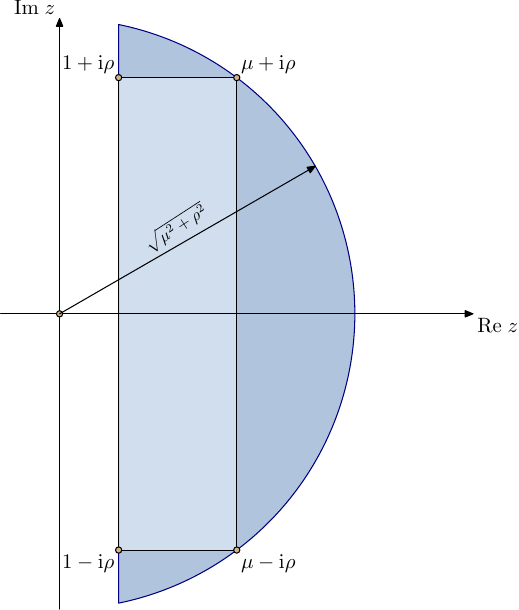}
\end{minipage}
\begin{minipage}{0.35\textwidth}
\includegraphics[width=\textwidth]{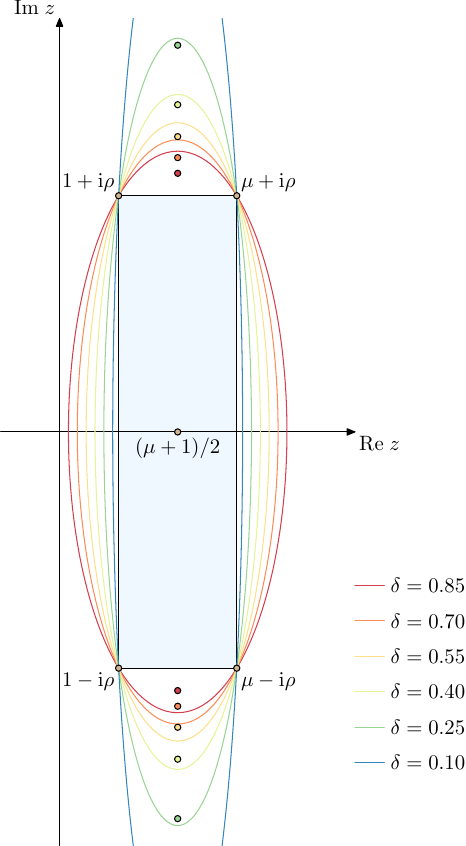}
\end{minipage}
\caption{$\Omega$ is defined by $\mu = 3$ and $\rho = 4$. Left: Disk enclosing $\Omega$ ($a = 17$ and $r = 4 \sqrt{17}$). Center: Disk-segment circumscribing $\Omega$. Right: Family of ellipses $E(c,d,b)$ circumscribing $\Omega$. The dot on the real axis is the center of all ellipses. The other dots are the foci of the ellipses.}
\label{fig:geomdisksegment} 
\end{figure}

\paragraph{Ellipse bound.}

Let $E$ be an ellipse that circumscribes\footnote{\revised{We mention that while we consider ellipses bounding rectangles
containing the field of values of $A$, ellipses bounding rectangles bounding the eigenvalues of $A$ where studied in \cite{Fischer99} in the context
of GMRES convergence.}}
$\Omega$. It is parametrized by $(c, d, a) \in \mathbb C^3$, with $c$ its center, $d$ its focal distance and $a$ its semi-major axis.  For an ellipse $E$ that circumscribes our rectangle $\Omega$, we first choose  $c  = (\mu +1)/2$ and then determine $a$ and $d$. The Cartesian equation for $E$ in $\mathbb R^2$ is 
\[
\frac{(x-c)^2}{\alpha^2} + \frac{y^2}{\beta^2} \leq 1, \text{ for some } \alpha, \beta \in \mathbb R^+.
\]
A single equation ensures that the vertices $(c \pm \frac{\mu - 1}{2}, \pm \rho)$ of $\Omega$ (once converted into Cartesian coordinates) are on $\partial E$: 
\[
\frac{(\mu-1)^2}{4\alpha^2} + \frac{\rho^2}{\beta^2} = 1 \quad \Leftrightarrow \quad \beta = \rho \left[1 - \frac{(\mu - 1)^2}{4 \alpha^2}  \right]^{-1/2} \text{ and }  \alpha > \frac{\mu-1}{2} \cdot
\]
The circumscribing ellipses are parametrized by $\alpha > \frac{\mu-1}{2}$. The origin is outside the ellipse if  $\alpha < \frac{\mu+1}{2}$. To go back to complex notation, there are two cases. If $\alpha > \beta$, the major semi-axis is horizontal so $a = \alpha$ and $d = \sqrt{\alpha^2 - \beta^2}$. If $\alpha < \beta$, the major semi-axis is vertical so $a =i  \beta$ and $d =i \sqrt{\beta^2 - \alpha^2}$. An illustration of this second case can be seen in Figure~\ref{fig:geomdisksegment}--right.  

The well-known bound $K_k(E)$ for the min-max problem over the ellipse $E$ comes from maximizing a (near-optimal) scaled Chebyshev polynomial given \textit{e.g.}, in \cite[equation (6.119)]{zbMATH01953444}: 
\begin{equation}
\label{eq:boundellipse}
K_k(E) \leq \frac{C_k\left(\frac{a}{d}\right)}{\left|C_k\left(\frac{c}{d}\right)\right|}, \text{ reached by } \hat C_k: \revised{z} \mapsto \frac{C_k\left(\frac{c-z}{d}\right)}{C_k\left(\frac{c}{d}\right)} \text{ at } z = c+a, 
\end{equation} 
where $C_k$ is the Chebyshev polynomial of the first kind of degree $k$  and it has implicitly been assumed that $a/d$ is real, either because $a$ and $d$ are both real (\textit{horizontal ellipse}) or because they are both imaginary (\textit{vertical ellipse}). The asymptotic convergence rate, given in \cite[(6.121)]{zbMATH01953444}, is  
\begin{equation}
\label{eq:ellasymptotic}
\frac{C_k\left(\frac{a}{d}\right)}{\left|C_k\left(\frac{c}{d}\right)\right|} \approx   \left| \frac{ a + \sqrt{a^2 - d^2}}{c + \sqrt{c^2 - d^2}}\right| ^k.
\end{equation}

\begin{figure}[htb]
\includegraphics[width=0.95\textwidth]{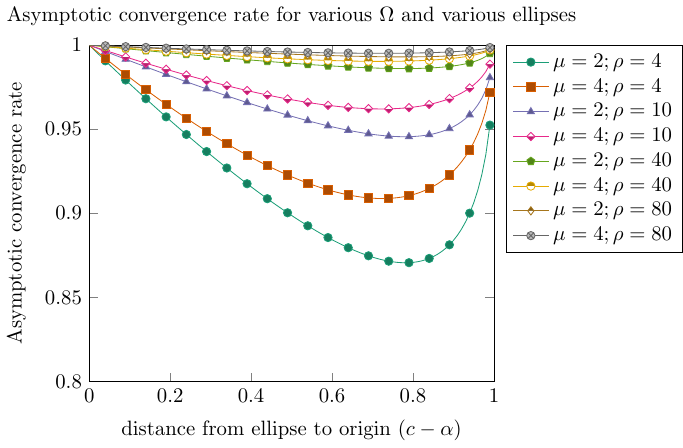}
\caption{For various choices of $\Omega$ parametrized by $\mu$ and $\rho$, asymptotic convergence rate $| ({ a + \sqrt{a^2 - d^2}})/({c + \sqrt{c^2 - d^2}})|$ from \eqref{eq:ellasymptotic} with respect to the distance from the ellipse to zero. For each $\Omega$ there is an optimal ellipse that gives the best (\textit{i.e.}, lowest) convergence rate.} 
\label{fig:ellasymptotic} 
\end{figure}

In Figure~\ref{fig:ellasymptotic} we plot the asymptotic convergence rate $| ({ a + \sqrt{a^2 - d^2}})/({c + \sqrt{c^2 - d^2}})|$  with respect to $c - \alpha$, the distance from the origin to the ellipse, for various choices of $\Omega$. Generally speaking, the asymptotic convergence rate deteriorates (\textit{i.e.}, increases) when $\Omega$ gets larger. We observe that in each case there is an optimal value of $\alpha$ which minimizes the asymptotic convergence rate. This reflects the fact that there is a trade-off between the ellipse not becoming too tall and it not becoming too close to the origin. In Figure~\ref{fig:ellcv}, the bound given by \eqref{eq:boundellipse} is plotted with respect to the polynomial order $k$ for various choice of ellipses (parametrized by the distance from the ellipse to $0$) and two values of $\rho$. When $\Omega$ gets taller, 
\revised{the value of
$K_k(\Omega)$ 
is reduced at a slower pace when $k$ grows}:
for example, at iteration $150$, the best bound for $K_{150}$ is $0.12$ when $\rho = 40$ whereas it is $2.2 \cdot 10^{-4}$ when $\rho = 10$. These best bounds are reached by the ellipse that has the optimal convergence rate. 
From now on when we consider the bound that comes from enclosing $\Omega$ in an ellipse, we choose the near-optimal ellipse found by selecting the $\alpha$ that minimizes the asymptotic convergence rate  out of $100$ values evenly spread out over $](\mu-1)/2,(\mu+1)/2[$. 

\begin{figure}
\includegraphics[width=0.95\textwidth]{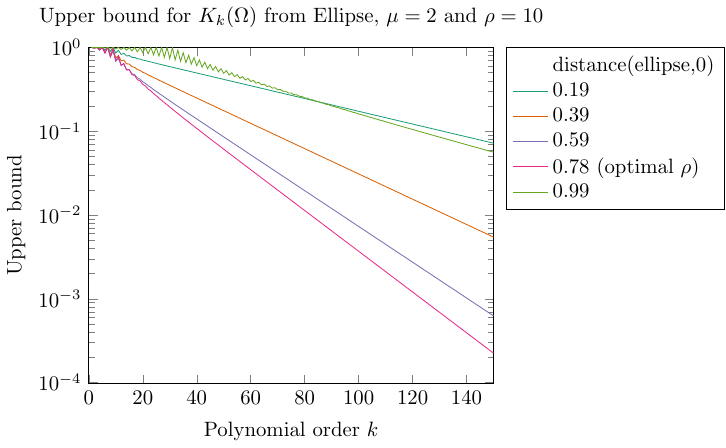}
\\
\includegraphics[width=0.95\textwidth]{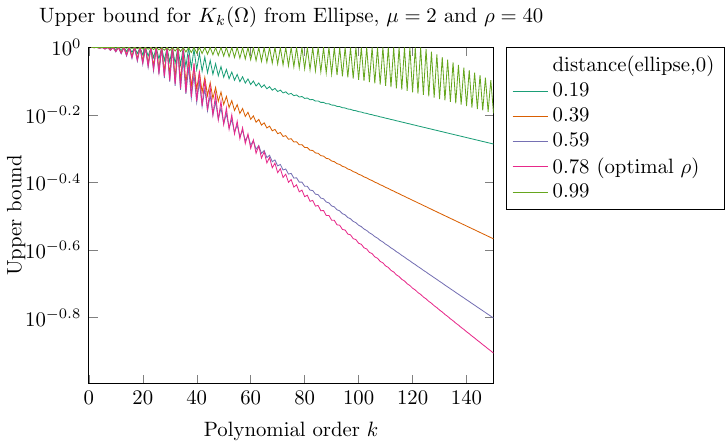}
\caption{Bound for $K_k(\Omega)$ computed using \eqref{eq:boundellipse} for different choices of the enclosing ellipse including the one that gives the optimal convergence rate. Top: $(\mu, \rho) = (2, 10)$ --   Bottom: $(\mu, \rho) = (2, 40)$. The vertical axis is not the same for the two plots. } 
\label{fig:ellcv} 
\end{figure}

\begin{remark}
Figure~\ref{fig:hatCk} shows a contour plot of $|\hat C_k|$ for $k = 1,\,2,\,3$ and for the optimal ellipse. The plot is over the rectangle that encloses the ellipse (which in turn encloses $\Omega$). We observe that the functions have largest magnitude toward the corners but these points are not in the ellipse (or in $\Omega$). To improve the bound \eqref{eq:boundellipse}, we also tried numerically maximizing $|\hat C_k|$ over $\Omega$ instead of over the ellipse. This does not significantly improve the result so we do not report results here. \revised{The bound obtained with this technique is larger for $k=3$ than for $k=2$.} The fact that $\hat C_k$ is not in general optimal goes back to \cite{zbMATH00012879}. 
\end{remark}

\begin{figure}
\begin{center}
\begin{minipage}{0.43\textwidth}
\[k=1; \quad K_1(\Omega) \leq  3.7060 \] 
\includegraphics[width=\textwidth]{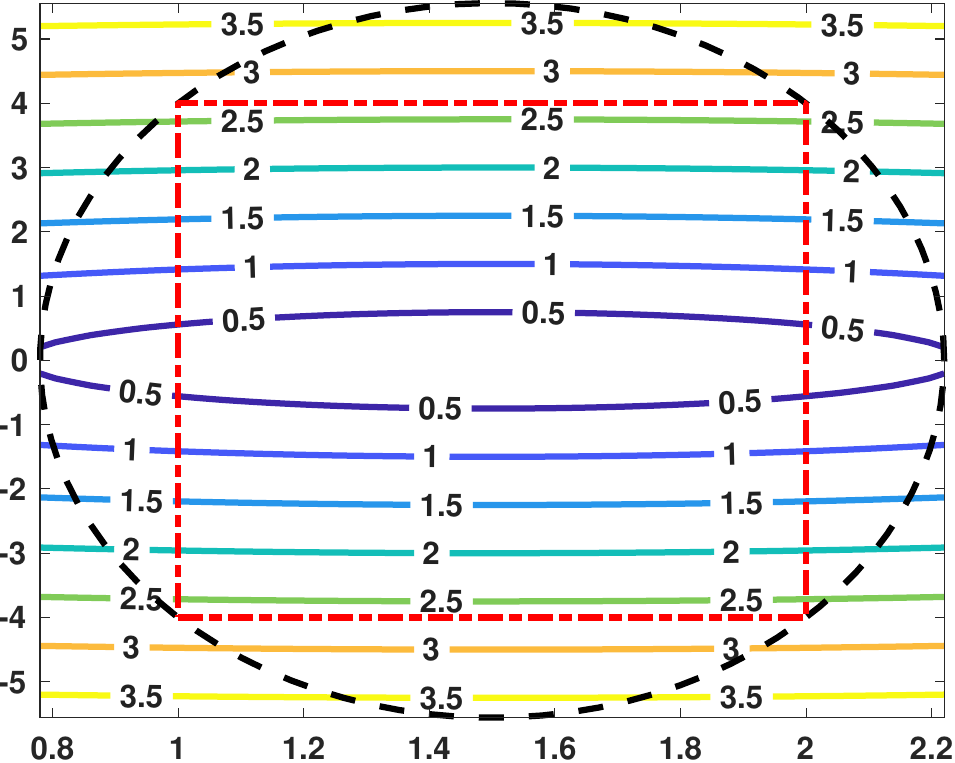}
\end{minipage}
\begin{minipage}{0.43\textwidth}
\[k=2; \quad K_2(\Omega) \leq  0.9007   \] 
\includegraphics[width=\textwidth]{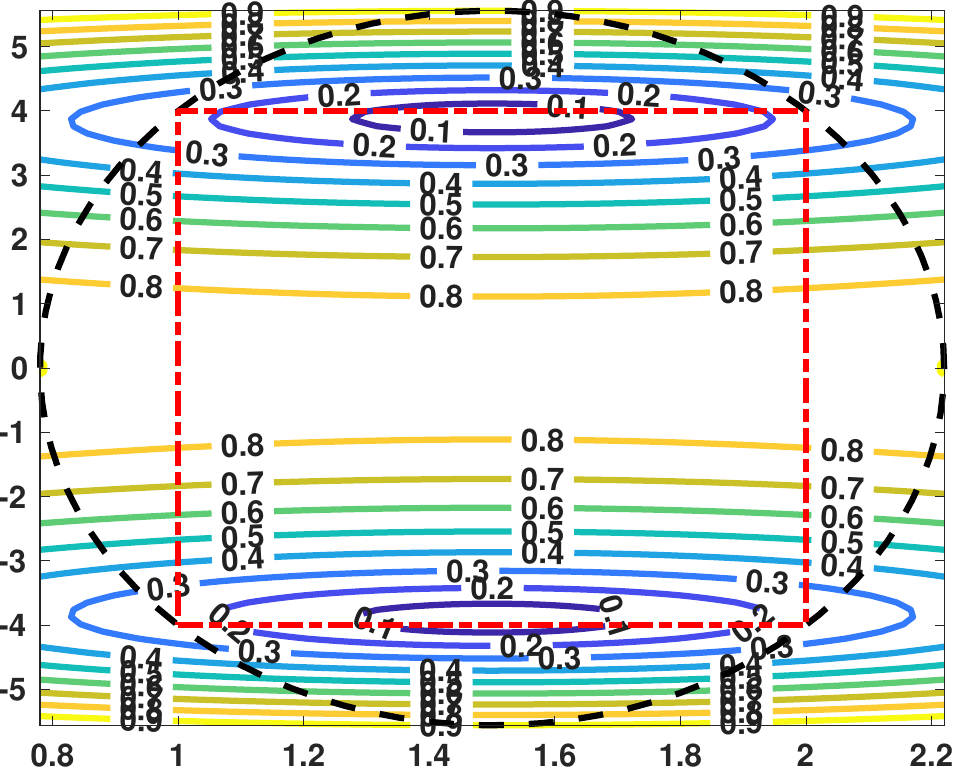}
\end{minipage}
\\
\begin{minipage}{0.43\textwidth}
\[k=3; \quad K_3(\Omega) \leq   1.2011 \] 
\includegraphics[width=\textwidth]{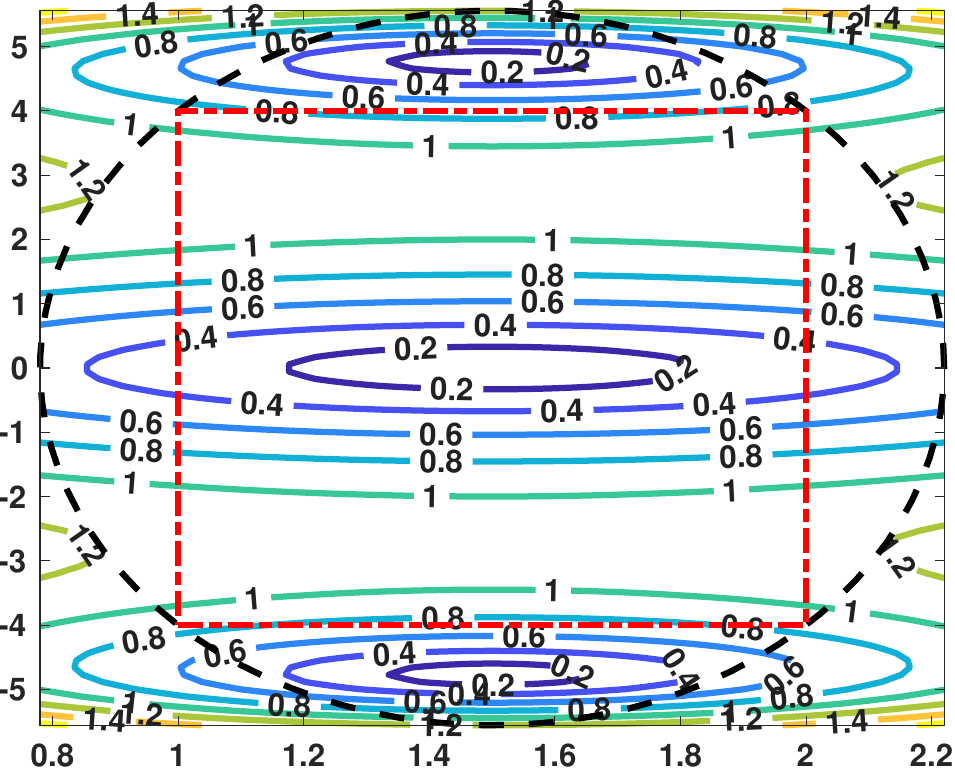}
\end{minipage}
\begin{minipage}{0.43\textwidth}
\[k=4; \quad K_4(\Omega) \leq  0.6960\]
\includegraphics[width=\textwidth]{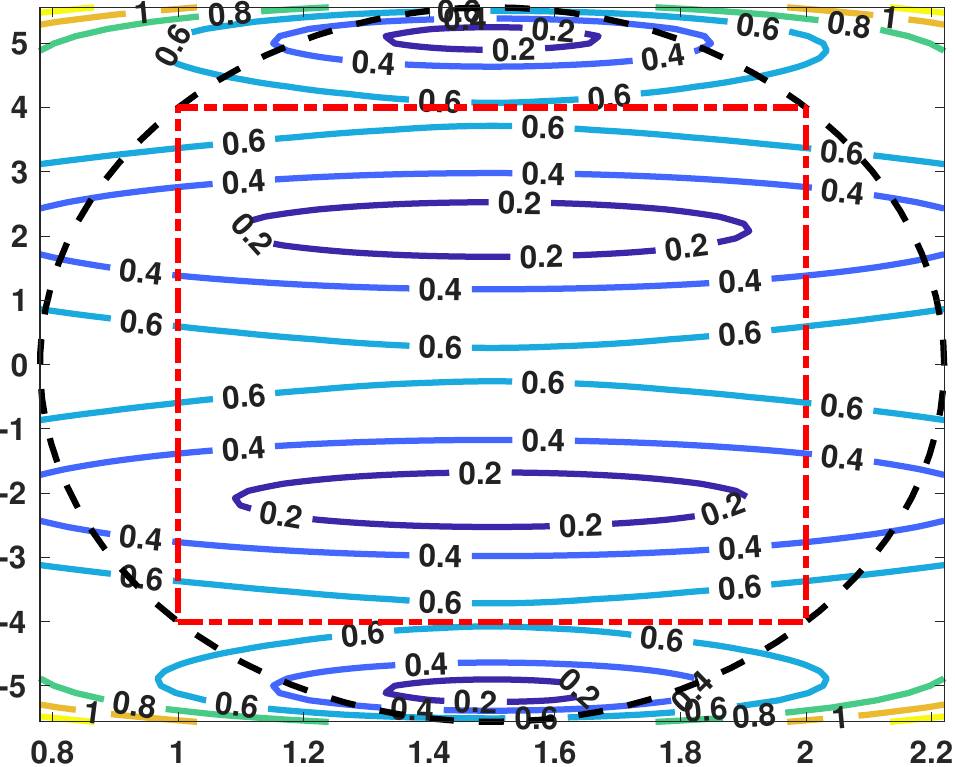}
\end{minipage}
\end{center}
\caption{\revised{Contour plot (in the imaginary plane) of $|\hat C_k|$, the modulus of the polynomial from which the bound in \eqref{eq:boundellipse} is obtained for polynomial orders $k = 1,\,2,\,3,\,4$. The red dot-dashed rectangle shows $\Omega$ (for $\mu =2$ and $\rho = 4$). The black dashed ellipse is the optimal ellipse from which $\hat C_k$ is defined. The resulting upper bound for $K_k(\Omega)$ is also reported. The real and imaginary axes have different scales. }} 
\label{fig:hatCk}
\end{figure}

\paragraph{Conformal Mapping bound.}

In \cite[Proof of Lemma 2.2]{zbMATH05029264}, as a first step in the proof of~\eqref{eq:Beck}, a lower and upper bound for $K_k(C)$ is proved for any convex compact set \revised{$C$} such that $0 \not\in C$ that contains at least two elements. Letting $\phi$ denote the Riemann conformal mapping from $\overline{ \mathbb C} \setminus C$ onto the exterior of the closed unit disk with $\phi(\infty) = \infty$, the result is that, for any $k \geq 1$,
\begin{equation}
\label{eq:boundconf}
\gamma^k \leq  K_k(C) \leq \min \left\{2 + \gamma , \frac{2}{1 - \gamma^{k+1}}  \right\} \gamma^k \leq 3 \gamma^k; \text{ where } \gamma := \frac{1}{\phi(0)} \cdot
\end{equation}

This is a tight bound since a factor less than $3$ separates the left- and right- hand sides. For the rectangle $\Omega$ we have evaluated $\gamma$ numerically thanks to the (Schwarz-Christoffel) SC-toolbox \cite{zbMATH01757217} for Matlab. This takes only four lines of code, as follows.

\begin{verbatim}
Omega = polygon([1+1i*rho 1-1i*rho mu-1i *rho mu+1i*rho]);
M = extermap(Omega);
invM = inv(M); %Map from the exterior of Omega to the interior of the unit disk
gamma = abs(invM(0))  
\end{verbatim}
Indeed, \verb+invM+  maps to the interior of the unit disk. To map to the exterior of the unit disk and obtain $\phi$, \verb+invM+  should be composed with $z \mapsto 1/z$. This inverse cancels out with the inverse in the definition of $\gamma$. Schwarz-{Christoffel} conformal mapping applies to any polygon $p$ so this technique is more general than the solution of the min-max problem on a rectangle.  

\paragraph{Faber polynomial bound.}
As a final bound we include part of estimate \cite[(11)]{zbMATH02182054}. For any $k \geq 1$
\begin{equation}
\label{eq:boundFaber}
K_k(C) \leq \frac{2}{|F_k^E(0) |}; \text{ where } F_k^E \text{ is the } k\text{-th Faber polynomial for $C$}. 
\end{equation}
 The Faber polynomial is the polynomial part of the Laurent expansion at infinity of $\phi^n$. When $C$ is polygonal (\textit{e.g.}, rectangular), it can also be computed by the SC-toolbox as follows. 
\begin{verbatim}
F = faber(p,k); %k is the polynomial degree
\end{verbatim}

\paragraph{Comparison between all the bounds.}

Figures~\ref{fig:mu2rho4}--\ref{fig:mu2rho40}
show the values given by all bounds (`Disk' \eqref{eq:boundOS}, `Ellipse' \eqref{eq:boundellipse} , `Disk-segment' \eqref{eq:Beck} , `Conformal map' \eqref{eq:boundconf} and `Faber' \eqref{eq:boundFaber}) with respect to the polynomial order $k$. If the value returned by a bound is larger than $1$, it is set to $1$. If the value returned by a bound at polynomial order $k$ is larger than at $k-1$, we replace it by the value at $k-1$. This way there are no distracting oscillations in the plots. 

Each Figure corresponds to a different choice of $\Omega$. Figure~\ref{fig:mu2rho4} corresponds to $\Omega = [1,2] \times [-4,4]$. It is observed that for $k=100$ the `Conformal map' and `Faber' bounds give the best estimate, followed by `Disk-segment' and `Ellipse' (which are several orders of magnitude larger). The `Disk' bound is very pessimistic at $k=100$. However, to make this picture more complete, the bottom of Figure~\ref{fig:mu2rho4} shows the same data up to the order $k=10$. It becomes apparent that, the bounds `Disk-segment', `Conformal Map' and `Faber', are lower than $1$ only for $k \geq 5$ or $6$. This is a consequence of the multiplicative constant. The bounds `Disk' and `Ellipse' do provide a bound even for these small values of $k$. In fact, `Ellipse' is the best bound up to $k=9$. 

Figure~\ref{fig:mu33rho4} is for $\Omega = [1,33] \times [-4,4]$, the behavior is very similar to that observed in Figure~\ref{fig:mu2rho4} except the values of the bounds: when $\mu=33$, the best bound after $k=100$ is approximately $10^{-4}$, versus $10^{-9}$ when $\mu = 2$.   

Figure~\ref{fig:mu2rho10} is for $\Omega = [1,2] \times [-10,10]$. We can again make similar comments on the comparison between the bounds and the best final value is again approximately $10^{-4}$. The `Ellipse' bound is the best for $k \leq 15$.  

Finally, Figure~\ref{fig:mu2rho40} is for $\Omega = [1,2] \times [-40,40]$. This time the final value is significantly worse: the best estimate with $k=100$ is larger than $10^{-1}$. Making the rectangle taller significantly worsens the bound.  

The takeaway is that bounds that approximate correctly the asymptotic convergence rate (`Conformal map', `Faber') give the best result after a large number of iterations but `Disk' and `Ellipse' are the only ones that are guaranteed to be informative at every iteration (even the first). `Disk-segment' gives the correct convergence rate but on the disk-segment (which is sometimes much larger than $\Omega$). It does have the advantage over `Conformal Map' and `Faber' that the formula is more direct.  
\begin{figure}
\includegraphics[width=0.9 \textwidth]{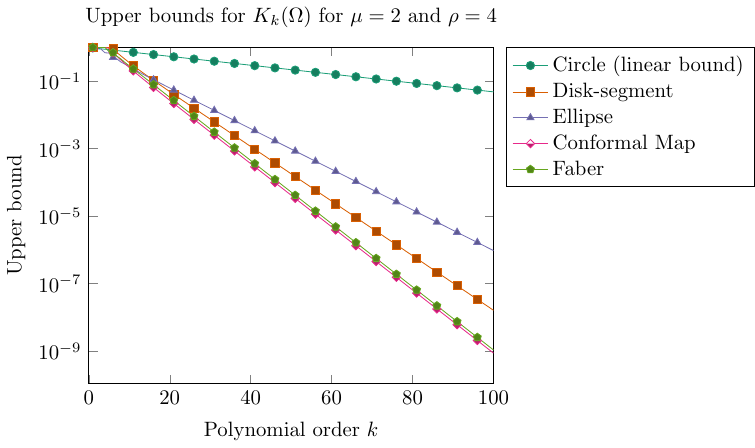}
\\
\includegraphics[width=0.9 \textwidth]{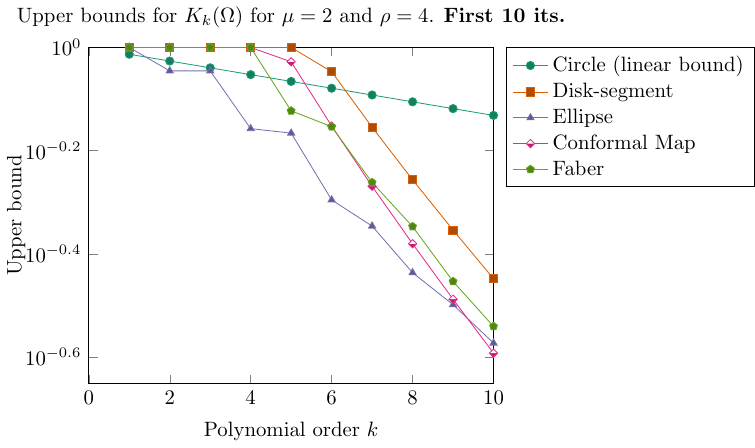}
\caption{Comparison between all the bounds for $\mu = 2$ and $\rho = 4$. The bottom plot is a zoom on the first 10 iterations.} 
\label{fig:mu2rho4}
\end{figure}

\begin{figure}
\includegraphics[width=0.9 \textwidth]{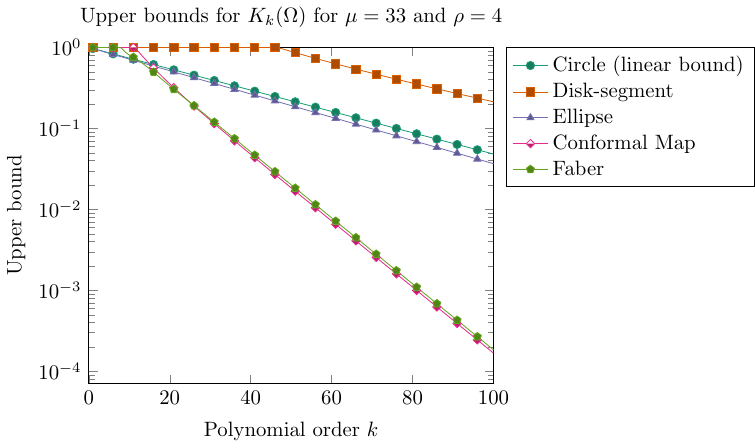}
\caption{Comparison between all the bounds for $\mu = 33$ and $\rho = 4$.}
\label{fig:mu33rho4}
\end{figure}

\begin{figure}
\includegraphics[width=0.9 \textwidth]{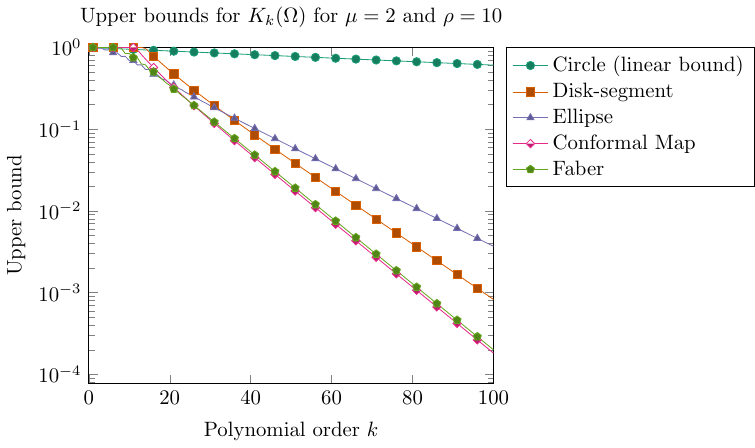}
\\
\includegraphics[width=0.9 \textwidth]{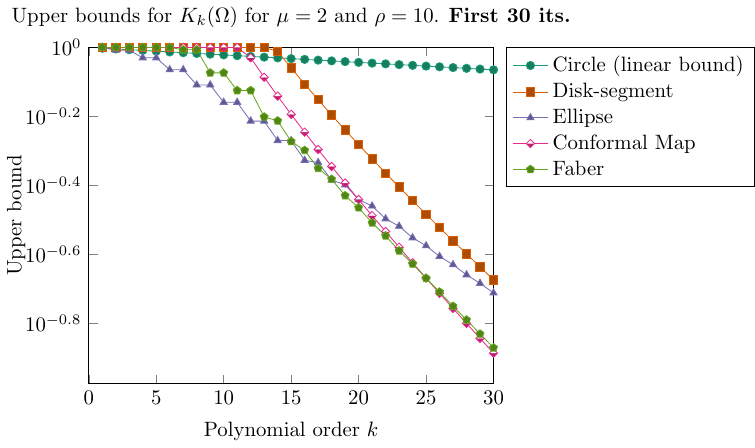}
\caption{Comparison between all the bounds for $\mu = 2$ and $\rho = 10$. The bottom plot is a zoom on the first 30 iterations.} 
\label{fig:mu2rho10}
\end{figure}

\begin{figure}
\includegraphics[width=0.9 \textwidth]{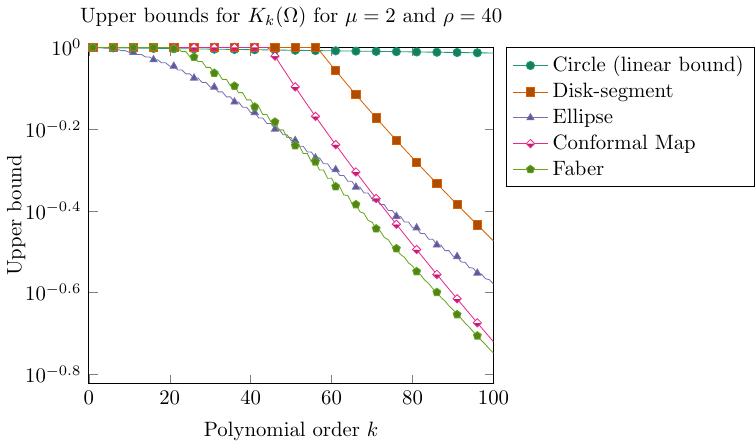}
\caption{Comparison between all the bounds for $\mu = 2$ and $\rho = 40$. After 100 iterations, the best of the bounds is still above $10^{-1}$.}
\label{fig:mu2rho40}
\end{figure}

\section{Numerical Results} \label{sec:numerical}
In this section, the problem considered is the convection-diffusion-reaction problem posed in \revised{${\cal D} = [-1,1]^2$}.  The strong formulation of the problem is: 
\begin{align*}
c_0 u + \operatorname{div}(\mathbf a u) - \operatorname{div} (\nu \nabla u) &= f \text{ in } \revised{\cal D},\\ 
u &= 0 \text{ on } \partial \revised{\cal D}.
\end{align*}

The variational formulation is: 
Find $u \in H^1_0(\revised{\cal D})$ such that 
\[
\underbrace{\int\revised{\cal D} \left(\left(c_0 + \frac{1}{2} \operatorname{div} \mathbf a \right) uv+ \nu \nabla u \cdot \nabla v \right)}_{\text{symmetric part}}  + \underbrace{\int_\revised{\cal D} \left(\frac{1}{2} \mathbf a \cdot \nabla u v - \frac{1}{2} \mathbf a \cdot \nabla v u \right)}_{\text{skew-symmetric part}}  = \int_\revised{\cal D} fv, 
\]
for all $ v \in H^1_0(\revised{\cal D})$. The reaction coefficient $c_0 >0$ and viscosity $\nu >0 $ are assumed to be constant over $\revised{\cal D}$. The right hand side and convection field are chosen as 
\[
f(x,y) =  \operatorname{exp} (-2.5(x ^2 + (y +0.8)^2)), 
\text{ and } \mathbf a(x,y) = \eta \pi \begin{pmatrix} - y - 0.8 \\  x  \end{pmatrix} \text{ with } \eta = 100. 
\]
It can be remarked that $ \operatorname{div} \mathbf a =0$. For the parameters in the problem, we set
\[
c_0=1, \, \nu = 1, \text{ and } \eta =100.
\]

\begin{figure}
\begin{center}
\includegraphics[width=0.49\textwidth]{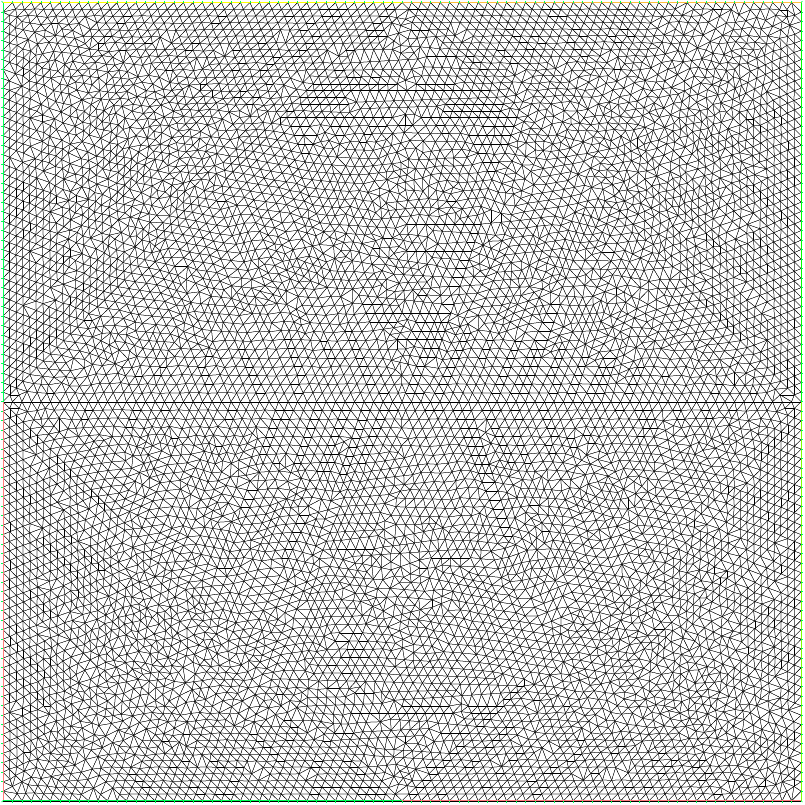}
\includegraphics[width = 0.49 \textwidth]{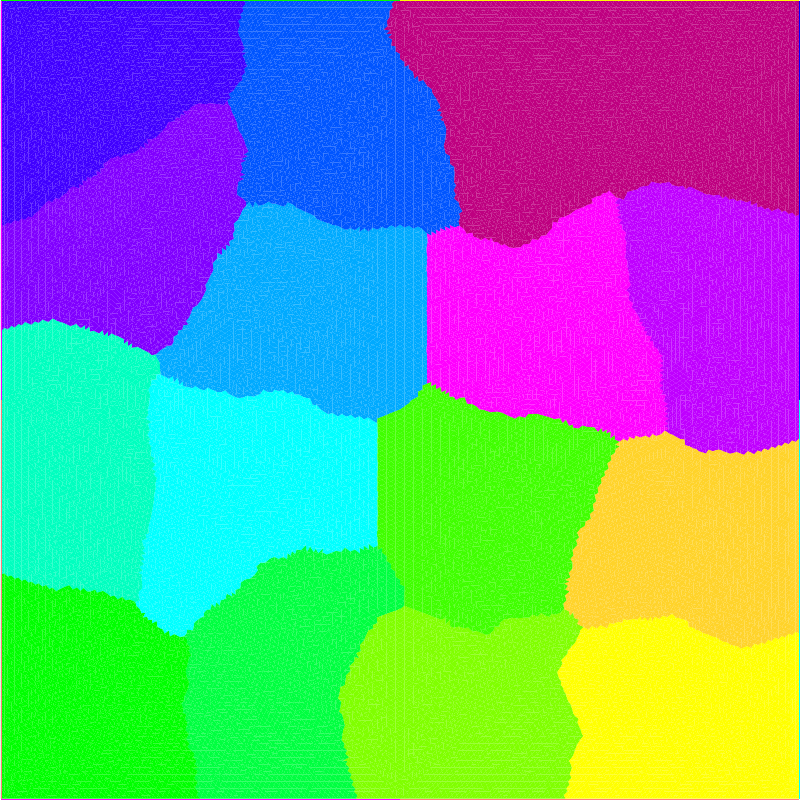}
\end{center}
\caption{Left: Mesh consisting of 8643 vertices and 16948 triangles. Right: Partition into 16 subdomains.}
\label{fig:meshnloc21}
\end{figure}


The problem is discretized by Lagrange $\mathbb P_1$ finite elements on a triangular mesh with 8643 vertices and 16948 triangles (see Figure~\ref{fig:meshnloc21} -- left). GMRES is implemented in Octave while the finite element matrices and right hand side are assembled by FreeFem++ \cite{MR3043640}.  All iteration counts correspond to the number of iterations needed to reach $\|\br_i\|_\bH < 10^{-10} \|\bb \|_\bH$ starting from a zero initial vector. 
The Dirichlet boundary condition has been enforced by elimination. 
Let $(\phi_i)_{1 \leq i \leq n}$ denote the $\mathbb P_1$ finite element basis corresponding to the mesh. The problem matrix splits into
\[
\bA = \bM + \bN, \text{ with } \bM \text{ spd and } \bN \text{ skew-symmetric},
\]
where the entries of $\bM$ and $\bN$ are 
\[
\bM_{ij} = \int_\revised{\cal D} \left(c_0 \phi_i \phi_j + \nu \nabla \phi_i \cdot \nabla \phi_j \right), \text{ and } \bN_{ij}= \eta \int_\revised{\cal D} \left(\frac{1}{2} \revised{\mathbf {a}} \cdot \nabla \phi_i \phi_j - \frac{1}{2} \revised{\mathbf{a}} \cdot \nabla \phi_j \phi_i \right). 
\] 
The positive definiteness of $\bM$ is guaranteed by the assumption that $c_0$ and $\nu$ are positive. 
This is the same setup as in \cite{zbMATH07931226}, and with a change of variables as in \cite{zbMATH07860856}. 

\paragraph{Choice of preconditioner and inner product.}

As a preconditioner $\bH$, we select a domain decomposition (DD) preconditioner based on a partition of the mesh into $N = 16$ subdomains (as shown in Figure~\ref{fig:meshnloc21} (right)). 
In detail, $\HDD$ is the Additive Schwarz domain decomposition method with the GenEO coarse space \cite{2011SpillaneCR,spillane2013abstract}. The condition number of the resulting preconditioned operator is bounded by
\[
\lambda_{\min}(\bH\bM) \geq \left(1 + \frac{k_0}{\upsilon}\right)^{-1}; \quad \lambda_{\max}(\bH\bM) \leq k_0; \text{ and }\kappa(\bH \bM) \leq k_0 \left(1 + \frac{k_0}{\upsilon}\right), 
\] 
where $k_0$ ($=3$ here) denotes the maximal number of subdomains that each mesh element belongs to \cite[Theorem 4.40]{spill2014} and  $\upsilon$ is a parameter that has been set to $0.15$. The constant in the bound does not depend on the total number $N$ of subdomains or the mesh parameter $h$. 

The preconditioner is spd. It is applied on the right ($\bH_R = \bH$ and $\bH_L = \matid$) and the weight is $\bW = \bH$. 

\paragraph{Deflation Operator.}

We aim to illustrate the convergence results in Theorems~\ref{th:gevpHN} and~\ref{th:gevpMinvN} for the two spectral deflation spaces.  Once $\bM$, $\bN$ and the domain decomposition operators that go into $\bH$ have been assembled by FreeFem++, they are imported into Octave. The preconditioner $\bH$ is never assembled into a matrix. Then the generalized eigenvalue problem $ \bN \bx = \lambda \bM \bx$ or $ \bN \bx = \lambda \bH^{-1} \bx$  is partially solved by \textit{eigs}: the eigenpairs corresponding to the eigenvalues of largest magnitude are approximated. Following Definition~\ref{def:PDQD}, the deflation operator is fully defined by the choice of $\bY$ and $\bZ$. We give more detail below.   

\paragraph{First deflation space (Theorem~\ref{th:gevpMinvN}).}
Solve \eqref{eq:gevpMinvN} (\textit{i.e.}, $\bN \bz^{(j)} = \lambda_{j} \bM \bz^{(j)}$) and order the eigenpairs in non-increasing \revised{modulus} order. Let $m \in \mathbb N$ be the desired dimension of the deflation space. Assuming that $m$ is even, let $\bZ$ be 
\[
\bZ := \left[\Re(\bz^{(1)}) \,|\, \Im(\bz^{(1)})  \,|\, \Re(\bz^{(3)}) \,|\, \Im(\bz^{(3)}) \,|\, \dots \,|\,  \Re(\bz^{(m-1)}) \,|\, \Im(\bz^{(m-1)}) \right], 
\] 
where it is meant that the vectors are the columns of $\bZ$.
 We could equivalently have set $\bZ = \left[ \bz^{(1)} \, | \,  \bz^{(2)}\, | \, \dots \, | \, \bz^{(m-1)} \, | \, \bz^{(m)} \right]$ but then the deflation operator is complex whereas the linear system and preconditioner are real, an unnecessary numerical expense. 
The reason both choices are equivalent is that $\bM$ is spd and $\bN$ is skew-symmetric so the eigenvalues,
\revised{which are purely imaginary},
come in complex conjugate pairs 
(with the possible exception being $0$),
with complex conjugate pairs of eigenvectors. Finally, we choose $\bY = \bH \bA \bZ$ and this completes the definition of the deflation operator $\bP_D$. 

\paragraph{Second deflation space (Theorem~\ref{th:gevpHN}).}

Solve \eqref{eq:gevpHN} (\textit{i.e.}, $\bN \revised{\bx}^{(j)} = \lambda_{j} \bH^{-1} \revised{\bx}^{(j)}$) and order the eigenvalues in non-increasing modulus order. Let $m \in \mathbb N$ be the desired dimension of the deflation space. Assuming that $m$ is even, let the matrix $\bY$ be 
\[
\revised{\bY = \left[\Re(\bx^{(1)}) \,|\, \Im(\bx^{(1)})  \,|\, \Re(\bx^{(3)}) \,|\, \Im(\bx^{(3)}) \,|\, \dots \,|\,  \Re(\bx^{(m-1)}) \,|\, \Im(\bx^{(m-1)}) \right],} 
\] 
where we mean that the vectors are the columns of $\bY$. For $\bZ$, we study three possible choices
\begin{itemize}
\item $\bZ = \bA^{-1} \bN \bY$, so that $\range (\bZ) = \range (\bA^{-1} \bN \bY) = \range (\bA^{-1} \bH^{-1} \bY)$. This satisfies the technical assumption $\bY = \bH \bA \bZ$ but it is unrealistic to apply $\bA^{-1}$. Note that $\bP_D$ can actually be assembled with $\bA \bZ = \bN \bY$ but the knowledge of $\bZ$ is needed to compute the solution to the original (non-deflated) problem. 
\item $\bZ = \bY$.
\item $\bZ = \bN \bY$.
\end{itemize}
The last two choices do not satisfy the technical assumption $\bY = \bH \bA \bZ$ but they are numerically feasible. 

\paragraph{Results.}

The spectra of both generalized eigenvalue problems are shown in Figure~\ref{fig:spectrum} where the modulus of the first $600$ 
\revised{eigenvalues} has been plotted. We observe that there is a lot of similarity between both curves. The table in Figure~\ref{fig:spectrum} also gives the numerical \revised{radii} of $\bM^{-1} \bN$ and $\bH \bN$ as well as some particular eigenvalues that confirm the closeness. \revised{These eigenvalues are used in the discussion below.} 

Figure~\ref{fig:fov} shows the $\bH$-weighted spectrum and the $\bH$-weighted field of values of the preconditioned operator $\bA \bH$. The axes have different scales: the field of values is actually tall and skinny. The rectangle $\Omega$ that encloses $FOV^\bH(\bA\bH)$ is 
\[
\Omega_1 = [0.21; \, 3.00] \times i [-48.9; \, 48.9] = 0.21 \left( [1 ; \, 14.3] \times i [\revised{-}233; \, 233]  \right). 
\]  
This is exactly $\Omega_1$ from Theorem~\ref{th:CP} (convergence without deflation).

Following the work in Section~\ref{sec:minmaxrec}, the best bound for the min-max problem on $\Omega_1$ for $k=200$ is $K_{200}(\Omega_1) \leq 0.85$. The corresponding GMRES residual bound is $\| \br_{200}\|_\bH / \| \br_0 \|_\bH \leq (1 + \sqrt{2}) K_{200}(\Omega_1) \leq 2.04 $. Unfortunately any bound larger than $1$ is not useful. Next we apply Theorem~\ref{th:gevpHN}. If $m=100$  vectors of $\bH\bN$ are deflated, the bound computed from $\lambda_{101}$ is $K_{200}( [0.21; \, 3.00] \times i [-11.7; \, 11.7]) \leq 0.07$. If $m=300$  vectors of $\bH\bN$ are deflated, the bound computed from $\lambda_{301}$ is $K_{200}( [0.21; \, 3.00] \times i [-7.1; \, 7.1]) \leq 0.02$. It is predicted that convergence improves when more deflation vectors are added. \revised{The estimates from Theorem~\ref{th:gevpMinvN} for deflation of eigenvectors of $\bM^{-1} \bN$ can also be computed in the same way (see below). Without deflation, Theorem~\ref{th:CP} introduces $\Omega_2 = [0.21; \, 3.00] \times  3.00 * i [-64.5; \, 64.5] = [0.21; \, 3.00] \times  i [-193.5; \, 193.5]$ for which the bound on the min-max problem resulting from the strategy in Section~\ref{sec:minmaxrec} is  $K_{200}(\Omega_2) \leq 0.98$. This is worse than previously, as expected, because $\Omega_2$ from Theorem~\ref{th:gevpMinvN} is a superset of  $\Omega_1$ above. 
}

\revised{
\begin{table}
\begin{center}
Theorem~\ref{th:gevpHN} and deflation of eigenvectors of  $\bH \bN$ (with $\bY = \bH\bA\bZ$)\\
\begin{tabular}{c|c|c|c|c}
&\multicolumn{2}{c|}{iteration $100$}& \multicolumn{2}{c}{iteration $180$}\\
\hline
                              & Bound& Experiment  & Bound& Experiment  \\ 
\hline
Without deflation             & $2.29$ & $0.013$         & $2.08$ & $2.3\cdot 10^{-4}$\\         
Deflation of $m= 100$ vectors & $1.02$ & $9.4\cdot10^{-4}$    & $0.29$ & $7.3 \cdot 10^{-7}$\\  
Deflation of $m= 300$ vectors & $0.43$ & $7.3 \cdot10^{-6}$    & $0.061$  & $9.7 \cdot 10^{-11}$ \\
\end{tabular}
\end{center}
\end{table}
\begin{table}
\begin{center}
Theorem~\ref{th:gevpMinvN} and deflation of eigenvectors of  $\bM^{-1} \bN$\\ 
\begin{tabular}{c|c|c|c|c}
&\multicolumn{2}{c|}{iteration $100$}& \multicolumn{2}{c}{iteration $180$}\\
\hline
                              & Bound& Experiment  & Bound& Experiment  \\ 
\hline
Without deflation             & $2.40$  &  $0.013$  & $2.38$  & $2.3\cdot 10^{-4}$  \\         
Deflation of $m= 100$ vectors & $2.20$  & $1.2 \cdot 10^{-3}$  & $1.98$  & $1.0 \cdot 10^{-6}$  \\  
Deflation of $m= 300$ vectors & $1.97$  &  $1.2 \cdot 10^{-5}$ & $0.96$  & $2.1 \cdot 10^{-10}$  \\
\end{tabular}
\caption{\revised{Numerical comparison of theoretical bounds and GMRES residuals. `Bound': relative residual bound as given in Theorem~\ref{th:gevpHN} (top) and Theorem~\ref{th:gevpMinvN} (bottom) at iterations $100$ and $180$. `Experiment': relative residual at iterations $100$ and $180$ of GMRES without deflation and with deflation of $m=100$ eigenvectors and $m=300$ eigenvectors of either $\bH \bN$ (top), or  $\bM^{-1} \bN$ (bottom). }}
\label{tab:theoryexperiment}
\end{center}
\end{table}
}
\revised{Finally, the convergence of GMRES on the proposed test case is analyzed. A first check is that the residuals satisfy the bounds from Theorems~\ref{th:gevpHN} and~\ref{th:gevpMinvN}. Explicit values are given in Table~\ref{tab:theoryexperiment}. It follows that the achieved residual norms are indeed below the bounds given in Theorems~\ref{th:gevpHN} and~\ref{th:gevpMinvN}. It is also observed that the theoretical bounds are not tight compared to the experiment to the extent that they do not give a good approximation of the residual. What the theoretical bounds do indicate is how deflation improves convergence and in this sense the theoretical analysis is successful. 
}

Convergence  curves are presented in Figure~\ref{fig:spectrumNM}. There, it is confirmed that the spectral deflation that we have proposed significantly accelerates convergence.  
In fact, without deflation or with $m=100$, there is no convergence in less than 200 iterations, 
in the sense that the relative residual norm is larger than $10^{-9}$,
while for $m=300$, convergence is achieved.
There is hardly any difference in convergence between deflating eigenvectors of $\bM^{-1} \bN$ and deflating eigenvectors of $\bH \bN$. There is also very little difference between the three variants in the second case. This is good news: the variants $\bY = \bZ$ and $\bN \bY = \bZ$ which lack the technical assumption $\bY = \bH \bA \bZ$ are just as efficient numerically.

\section{Conclusion}

We have presented an analysis of GMRES based on the Crouzeix-Palencia result that the field of values is a $(1+\sqrt 2)$ spectral set. The role of (left, right) preconditioning, weighting and deflation has been made explicit in the bounds. Two spectral deflation spaces were studied. Either the high-frequency eigenvectors of $\bM^{-1} \bN$ are deflated (as in \cite{zbMATH07931226}) or the high-frequency eigenvectors of $\bH \bN$ are deflated. Unless the inverse of $\bM$ is known, we would always recommend the second option. Indeed, with an iterative eigensolver, only applications of $\bH$ and $\bN$ are necessary (whereas with the first choice, the action of $\bM^{-1}$ must be computed, or approximated, many times). Theoretical results and numerical experiments show that deflation of these vectors indeed accelerates convergence of GMRES in terms of iterations.   
\\
\medskip

\begin{figure}[hbt]
\centering
\begin{minipage}{0.7\textwidth}
\includegraphics[width=0.7 \textwidth]{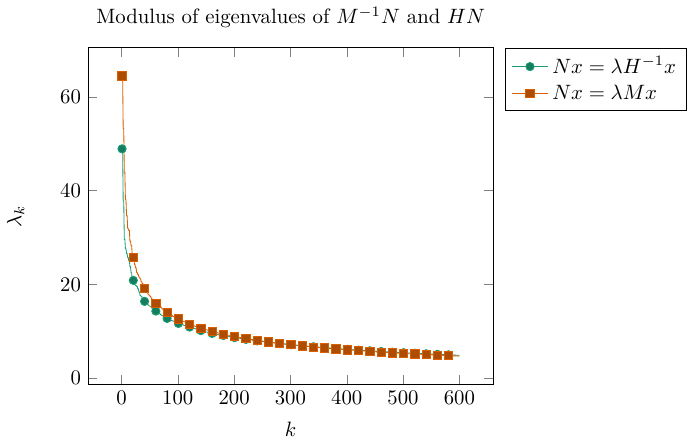}
\end{minipage}
\hspace*{-52mm}
\begin{minipage}{0.25\textwidth}
\begin{tabular}{ccc}
~~\\
Gevp & $\bN \bx = \lambda \bM \bx$ & $\bN \bx = \lambda \bH^{-1} \bx$ \\
$\rho = |\lambda_1|$ &  $  64.5$ & $48.9$ \\
$|\lambda_{101}|$ & $12.7$ & $11.7$ \\
$|\lambda_{301}|$ & $7.11$ & $ 7.10$ \\ 
$|\lambda_{600}|$ & $4.69$ & $4.87 $ \\ 
\end{tabular}
\end{minipage}
\caption{Upper part of the spectrum (first 600 eigenvalues) for the generalized eigenvalue problems (Gevps) that define the two deflation spaces.} 
\label{fig:spectrum} 
\end{figure}

\begin{figure}
\centering
\includegraphics[width=0.7 \textwidth]{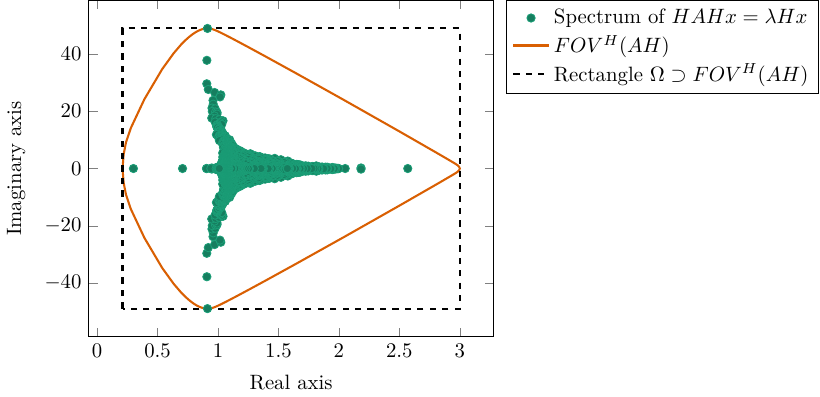}
\caption{ $\bH$-weighted spectrum and $\bH$-weighted field of values of the preconditioned operator without any deflation ($m=0$). (Remark that the axis have different scales.)}
\label{fig:fov}
\end{figure}

\begin{figure}[hbt]
\centering
\includegraphics[width=\textwidth]{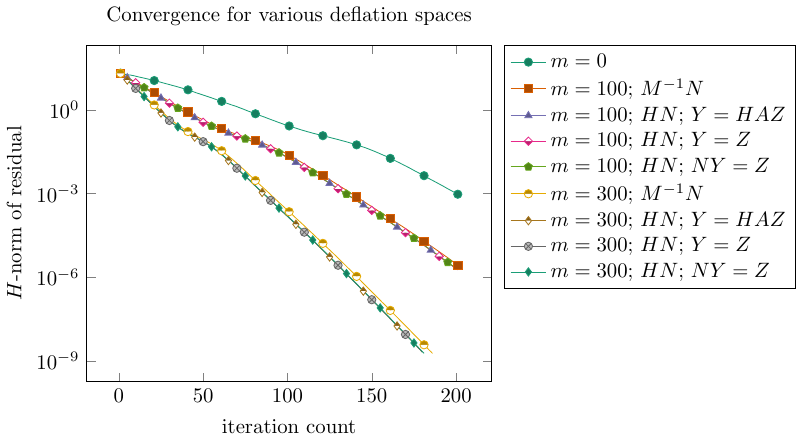}
\medskip 

{\footnotesize
\begin{tabular}{|c|c|c|c|c|}
\multicolumn{5}{c}{Summary of all iteration counts (and final residuals, either at convergence or after 200 iterations) }\\
\hline
 $m = 0$             &          \multicolumn{4}{c|}{$m=100$}            \\
\hline
                     &  $\bN \bx = \lambda \bM \bx$ & \multicolumn{3}{c|}{$\bN \bx = \lambda \bH^{-1} \bx$}   \\
\hline
                     &                              & $ \bY = \bH\bA\bZ$  & $\bY = \bZ$ & $\bN \bY = \bZ$       \\ 
\hline
$>200$               & $>200$                      & $>200$ & $> 200$ & $>200$           \\ 
{\tiny ($1.02 \cdot 10^{-3}$)} & {\tiny ($3.01 \cdot 10^{-6}$)}  & {\tiny( $2.05 \cdot 10^{-6}$)} & {\tiny($2.05 \cdot 10^{-6}$ )} & {\tiny ($2.09\cdot 10^{-6}$)} 
\\
\hline
\end{tabular}
\begin{tabular}{|c|c|c|c|}
\hline
        \multicolumn{4}{|c|}{$m=300$} \\ 
\hline
                      $\bN \bx = \lambda \bM \bx$ &  \multicolumn{3}{c|}{$\bN \bx = \lambda \bH^{-1} \bx$}  \\ 
\hline
                   &      $ \bY = \bH\bA\bZ$  & $\bY = \bZ$ & $\bN \bY = \bZ$  \\ 
\hline
  186    &  180  &  180  &  181 \\ 
{\tiny ($2.14 \cdot 10^{-9}$)} & {\tiny ($2.04 \cdot 10^{-9}$)}     &  {\tiny ($2.05\cdot10^{-9}$) }  & {\tiny ($1.85\cdot10^{-9}$) } \\
\hline
\end{tabular}
}
\caption{Convergence without a deflation space ($m=0$) and with deflation of $m=100$ or $m=300$. Both choices of deflation space are considered. In the case where the generalized eigenvalue problem is $\bN \bx = \lambda \bH^{-1} \bx$, all three variants for defining $\bZ$ with respect to $\bY$ are considered (and these curves are hardly distinguishable). The table summarizes the number of iterations needed to reduce the residual by a factor $10^{10}$ and, in parenthesis, either the converged residual or the residual after $200$ iterations when convergence is not yet achieved.}
\label{fig:spectrumNM} 
\end{figure}

\clearpage
\section*{Acknowledgements}

The authors thank Nick Trefethen for his help with the SC-toolbox for computing the convergence bounds based on Conformal mapping and Faber polynomials. 
 \revised{They also thank the two referees for their careful reading of the manuscript and their many comments and suggestions which helped
improved the presentation.}


\bibliographystyle{abbrv}
\bibliography{PolynomialBounds}

\end{document}